\documentclass[12pt,reqno]{article}

\usepackage{booktabs}

\usepackage{epsfig,amssymb,latexsym,amsmath,pifont,multicol,mathtools,verbatim,amsthm,float,lipsum}
\usepackage{caption} 
\usepackage[utf8]{inputenc}
\usepackage[T1]{fontenc}

\usepackage[varg]{txfonts}
\let\mathbb=\varmathbb

\usepackage[sans]{dsfont}

\usepackage[left=2cm, right=2cm, top=2cm, bottom=2cm]{geometry}

\usepackage[para,online,flushleft]{threeparttable}
\usepackage{multirow}
\usepackage[%
cal=euler,
bb=fourier,
scr=zapfc,
]
{mathalfa}

\usepackage[font=small,labelfont=bf]{caption}
\usepackage{subfigure}
\usepackage{graphicx}
\graphicspath{{Figures/}} 
\usepackage{acronym}
\usepackage{latexsym}
\usepackage{paralist}
\usepackage{wasysym}
\usepackage{xspace}
\usepackage{framed}
\usepackage{palatino,pxfonts}
\usepackage{authblk}
\usepackage{enumitem}
\usepackage[numbers,sort&compress]{natbib}

\usepackage[dvipsnames,svgnames]{xcolor}
\colorlet{MyBlue}{DodgerBlue!75!Black}
\colorlet{MyGreen}{DarkGreen!95!Black}

\usepackage{hyperref}
\hypersetup{
colorlinks=true,
linktocpage=true,
pdfstartview=FitH,
breaklinks=true,
pdfpagemode=UseNone,
pageanchor=true,
pdfpagemode=UseOutlines,
plainpages=false,
bookmarksnumbered,
bookmarksopen=false,
bookmarksopenlevel=1,
hypertexnames=true,
pdfhighlight=/O,
urlcolor=MyBlue!60!black,linkcolor=MyBlue!70!black,citecolor=DarkGreen!70!black, 
pdftitle={},
pdfauthor={},
pdfsubject={},
pdfkeywords={},
pdfcreator={pdfLaTeX},
pdfproducer={LaTeX with hyperref}
}

\numberwithin{equation}{section}  
\usepackage[sort&compress,capitalize,nameinlink]{cleveref}
\crefname{example}{Ex.}{Exs.}

\crefrangeformat{equation}{\upshape(#3#1#4)\textendash(#5#2#6)}

\usepackage{mathtools}
\usepackage[para,online,flushleft]{threeparttable}
\usepackage{multirow}

\usepackage{physics}
\usepackage{float}
\newcommand{\eps}{\varepsilon}
\DeclareMathOperator*{\argmin}{argmin}

\DeclareMathOperator{\cl}{cl}

\DeclareMathOperator{\zer}{zer}

\DeclareMathOperator{\indicator}{\iota}

\DeclareMathOperator{\dist}{dist}
\DeclareMathOperator{\dif}{d\!}

\DeclareMathOperator{\dom}{dom}

\DeclareMathOperator{\gr}{graph}

\DeclareMathOperator{\supp}{\mathtt{s}}
\DeclareMathOperator{\Var}{\mathsf{Var}}
\DeclareMathOperator{\Id}{Id}

\DeclareMathOperator{\prox}{\textsf{prox}}

\renewcommand{\iff}{\Leftrightarrow}
\renewcommand{\implies}{\Rightarrow}
\renewcommand{\emptyset}{\varnothing}
\newcommand{\eqdef}{\triangleq}
\newcommand{\wlim}{\rightharpoonup}

\newcommand{\scrA}{\mathcal{A}}
\newcommand{\scrB}{\mathcal{B}}
\newcommand{\scrC}{\mathcal{C}}
\newcommand{\scrD}{\mathcal{D}}

\newcommand{\scrF}{\mathcal{F}}

\newcommand{\scrK}{\mathcal{K}}

\newcommand{\scrN}{\mathcal{N}}

\newcommand{\scrP}{\mathcal{P}}

\newcommand{\scrS}{\mathcal{S}}
\newcommand{\scrT}{\mathcal{T}}
\newcommand{\scrU}{\mathcal{U}}

\newcommand{\scrW}{\mathcal{W}}
\newcommand{\scrX}{\mathcal{X}}
\newcommand{\scrY}{\mathcal{Y}}
\newcommand{\scrZ}{\mathcal{Z}}

\newcommand{\0}{\mathbf{0}}

\newcommand{\R}{\mathbb{R}}

\newcommand{\N}{\mathbb{N}}

\DeclareMathOperator{\NC}{\mathsf{NC}}
\DeclareMathOperator{\TC}{\mathsf{TC}}

\newcommand{\ball}{\mathbb{B}}

\newcommand{\opA}{\mathsf{A}}

\newcommand{\opM}{\mathsf{M}}
\newcommand{\opB}{\mathsf{B}}

\newcommand{\opF}{{\mathsf{F}}}
\newcommand{\opV}{{\mathsf{V}}}
\DeclareMathOperator{\VI}{VI}
\DeclareMathOperator{\HVI}{HVI}

\DeclareMathOperator{\gap}{\mathsf{Gap}}

\usepackage{algorithm}
\usepackage{algpseudocode}
\theoremstyle{plain}
\newtheorem{theorem}{Theorem}
\newtheorem{corollary}[theorem]{Corollary}
\newtheorem*{corollary*}{Corollary}
\newtheorem{lemma}[theorem]{Lemma}
\newtheorem{proposition}[theorem]{Proposition}

\theoremstyle{definition}
\newtheorem{definition}[theorem]{Definition}
\newtheorem*{definition*}{Definition}
\newtheorem*{problem*}{Problem}
\newtheorem{assumption}{Assumption}

\newcommand{\close}{\hfill{\footnotesize$\Diamond$}}

\theoremstyle{remark}
\newtheorem{remark}{Remark}
\newtheorem*{remark*}{Remark}
\newtheorem*{notation*}{Notational remark}
\newtheorem{example}{Example}

\numberwithin{theorem}{section}
\numberwithin{remark}{section}
\numberwithin{example}{section}

\DeclarePairedDelimiter{\inner}{\langle}{\rangle}

\title{Extragradient methods with complexity guarantees for hierarchical variational inequalities}
\date{\today}

\author[1]{\small Pavel Dvurechensky}
\author[2]{\small Meggie Marschner} 
\author[3]{\small Shimrit Shtern}
\author[2]{\small Mathias Staudigl}

\affil[1]{\footnotesize Weierstrass Institute for Applied Analysis and Stochastics, Anton-Wilhelm-Amo-Straße 39, 10117 Berlin, Germany\\
(\href{mailto:Pavel.Dvurechensky@wias-berlin.de}{Pavel.Dvurechensky@wias-berlin.de})}
\affil[2]{\footnotesize Mannheim University, Department of Mathematics, B6 26, 68159 Mannheim, Germany\\
(\href{mailto:mathias.staudigl@uni-mannheim.de}{m.marschner@uni-mannheim.de,m.staudigl@uni-mannheim.de})}
\affil[3]{\footnotesize Faculty of Engineering, Bar-Ilan University, Ramat-Gan, Israel\\
(\href{mailto:shterns@biu.ac.il}{shterns@biu.ac.il})}

\begin{document}

\maketitle

\begin{abstract}%
In the framework of a real Hilbert space we consider the problem of approaching solutions to a class of hierarchical variational inequality problems, subsuming several other problem classes including certain mathematical programs under equilibrium constraints, constrained min-max problems, hierarchical game problems, optimal control under VI constraints, and simple bilevel optimization problems. For this general problem formulation, we establish rates of convergence in terms of suitably constructed gap functions, measuring feasibility gaps and optimality gaps. We present worst-case iteration complexity results on both levels of the variational problem, as well as weak convergence under a geometric weak sharpness condition on the lower level solution set. Our results match and improve the state of the art in terms of their iteration complexity and the generality of the problem formulation.

\end{abstract}

\section{Introduction}
\label{sec:Intro}
%

Let $(\scrZ,\inner{\cdot,\cdot})$ be a real separable Hilbert space. We consider the problem of solving the hierarchical equilibrium problem, in which a \emph{hemi-variational inequality} ($\HVI$), also known as variational inequality of the second kind, is solved over the solution set of another $\HVI$:
\begin{equation}\label{eq:P}\tag{P}
\text{ Find }z^{\ast}\in\scrS_{2} \text{ s.t. } \inner{\opF_{1}(z^{\ast}),z-z^{\ast}}+g_{1}(z)-g_{1}(z^{\ast})\geq 0 \quad\forall z\in \scrS_{2}\eqdef \zer(\opF_{2}+\partial g_{2}). 
\end{equation}
The defining data are two monotone and Lipschitz continuous operators $\opF_1,\opF_2:\scrZ\to\scrZ$, and proper, convex lower semi-continuous functions $g_1, g_2$. We refer to the set $\scrS_{2}$ as the solution set of the lower-level problem. In the same vein, we will call $\scrS_{1}$ the set of solutions of \eqref{eq:P}. Obviously, these sets are nested in the sense that $\scrS_{1}\subseteq\scrS_{2}$. Hierarchical $\HVI$'s as in \eqref{eq:P} are prominent in mechanics \cite{goeleven2013variational} and optimal control \cite{Barbu93}. In finite dimensions, it has become a useful model template to study a plethora of problems in mathematical programming, machine learning, game theory, and signal processing \cite{malitsky2020golden,InertialFBF,Facchinei:2014aa}. Some motivating examples are given below.  

\begin{example}[Simple bilevel optimization]
\label{ex:SBP}
    Let $f_{1},f_{2}:\R^{n}\to\R$ be convex and continuously differentiable functions and $\opF_{1}=\nabla f_{1}$, $\opF_{2}=\nabla f_{2}$. In this setting, problem \eqref{eq:P} are the optimality conditions for the simple bilevel problem \cite{Dempe:2021aa,Solodov:2007aa} 
    \begin{equation*}
    \min_{x} f_{1}(x)+g_{1}(x),\quad \text{s.t.: } x\in \scrS_{2}=\argmin_{x\in\R^{n}}\{f_{2}(x)+g_{2}(x)\} . 
    \end{equation*}
This model has been studied in several very recent papers, e.g., in \cite{merchav2023convex}, which approach this problem via a disciplined proximal splitting ansatz, as well as in \cite{AminiYous19,doron2023methodology,yousefian2021bilevel,merchav2024fast,boct2025accelerating}. 
\end{example}

\begin{example}[Hierarchical variational inequalities]
If $g_{1}=0$ and $g_{2}$ is the convex indicator function of a non-empty, closed convex set $\scrK\subseteq\scrZ$, problem \eqref{eq:P} reduces to 
    \begin{equation*}
\text{Find }z^{\ast}\in\scrS_{2}\quad \text{s.t.: } \inner{\opF_{1}(z^{\ast}),z-z^{\ast}}\geq 0 \qquad \forall z\in\scrS_{2}=\zer(\opF_{2}+\NC_{\scrK}).
\end{equation*}
This is a challenging formulation of nested variational inequalities, which has been recently investigated in \cite{samadi_improved_2025,marschner2025tikhonov,alves2025inertial}. 
\end{example}

\begin{example}[Hierarchical jointly convex generalized Nash equilibrium problem]
\label{ex:HierarchyNE}
The recent work \cite{LamSagSIOPT25} studied a projection method for solving a challenging class of jointly convex generalized Nash equilibrium problems (GNEPs) featuring hierarchy and non-smooth data.  

\textbf{The lower-level Nash Equilibrium problem (NEP):}
Let $\nu \in \{1,\ldots,N\}$ represent the labels of the lower-level players, who engage in a Nash game in which each player's parameterized optimization problem is given by 
    \begin{equation*}
\min_{y^{\nu}\in\scrY^{\nu}}\{h^{\ell}_{\nu}(y^{\nu},y^{-\nu}))+\varphi^{\ell}_{\nu}(y^{\nu})\}.  
\end{equation*}
This defines a Nash game $\Gamma_{\rm low}\eqdef \{h^{\ell}_{\nu},\varphi_{\nu}^{\ell}\}_{1\leq\nu\leq N}$, for which the following assumptions shall be imposed: 
\begin{enumerate}
\item[(a)]  Player $\nu$'s feasible set $\scrY^{\nu}$ is a real separable Hilbert space; 
\item[(b)]  $h_{\nu}^{\ell}:\scrY^{\nu}\to\R$ is Fr\'{e}chet differentiable and convex with respect to $y^{\nu}$; 
\item[(c)] The game's pseudogradient $\opF_{2}(y)\eqdef (\nabla_{y^{1}}h^{\ell}_{1}(y), \ldots, \nabla_{y^{N}}h^{\ell}_{N}(y))$ is monotone on the separable Hilbert space  $\scrY\eqdef\prod_{\nu=1}^{N}\scrY^{\nu}$;
\item[(d)]  The functions $\varphi_{\nu}^{\ell}(\cdot)$ are proper, closed, and convex, with compact domain.
\end{enumerate}
It is well-known that the equilibria of the Nash game $\Gamma_{\rm low}$ can be characterized as solutions to the $\HVI$ with data $\opF_{2}$, and $
g_{2}(y^{1},\ldots,y^{N})\eqdef \sum_{\nu=1}^{N}\varphi^{\ell}_{\nu}(y^{\nu})
$ (see e.g. \cite{FacPan03}). Hence, $\scrS_{2}=\zer(\opF_{2}+\partial g_{2})$. Note that the above assumptions ensure that $\dom(g_{2})$ is a compact subset of $\scrY$. Furthermore, $\scrS_{2}$ is nonempty, convex, and compact. 

\textbf{The upper level GNEP:}
The upper-level problem is the Nash game with joint coupling constraint represented by the equilibrium set $\scrS_{2}$. Let $\mu \in \{1,\ldots,M\}$ denote the labels of upper-level players. Each upper-level player aims to solve the constrained optimization problem 
\begin{equation*}
\min_{x^{\mu}\in\scrX^{\mu}}\{h^{u}_{\mu}(x^{\mu},x^{-\mu})+\varphi^{u}_{\mu}(x^{\mu})\}\quad \text{s.t.: } (x^{\mu},x^{-\mu})\in\scrS_{2}.
\end{equation*}
As for the upper-level NEP, we make the following assumptions on the problem data:
\begin{enumerate}
\item[(e)] $h^{u}_{\mu}:\scrX^{\mu}\to\R$ is Fr\'{e}chet differentiable and convex with respect to $x^{\mu}$; 
\item[(f)] The operator $\opF_{1}(x)\eqdef ( \nabla_{x^{1}}h^{u}_{1}(x), \ldots,  \nabla_{x^{N}}h^{u}_{N}(x))$ is monotone on the product space $\scrX\eqdef\prod_{\nu=1}^{N}\scrX^{\nu}$; 
\item[(g)] $\varphi^{u}_{\mu}$ is proper closed convex with closed convex domain $\dom(\varphi^{u}_{\mu})$.
\end{enumerate}
The decision variables $x^{\mu}$ correspond to blocks $(y^{\nu})_{\nu\in\scrN_{\mu}}$, in which $\scrN_{\mu}\subseteq\{1,\ldots,N\}$ represent the lower-level players controlled by upper-level player $\mu$. In particular, $\scrX^{\mu}=\prod_{\nu\in\scrN_{\mu}}\scrY^{\nu}$. The entire hierarchical equilibrium problem thus belongs to our template \eqref{eq:P}, with $\scrS_{1}=\zer(\opF_{1}+\partial g_{1}+\NC_{\scrS_{2}})$, where 
$
g_{1}(x^{1},\ldots,x^{M})\eqdef\sum_{\mu=1}^{M}\varphi^{u}_{\mu}(x^{\mu}).
$
Note that if $M=1$ (i.e., there is only one player on the upper level), then we recover the problem of equilibrium selection in Nash games, which is an important problem in game theory \cite{yousefian2021bilevel,Benenati:2023aa}.
\end{example}

\subsection{Related works}
This paper focuses on a subclass of bilevel problems, usually termed as simple bilevel problems \cite{Dempe:2021aa}, in which a single decision variable is to be designed in order to solve the upper-level problem over the solution set of a lower-level problem. Starting with the seminal contributions \cite{Solodov:2007aa,Cabot2005}, Tikhonov regularization has become one of the dominant paradigms in the numerical solution of simple bilevel optimization problems.  In our $\HVI$ context, this idea leads to the definition of the data that combines both levels in \eqref{eq:P} as
\begin{equation}
\label{eq:VkGk_def}
    \opV_{k}\eqdef\opF_{2}+\sigma_{k}\opF_{1},\quad G_{k}\eqdef g_{2}+\sigma_{k}g_{1},
\end{equation}
where $k$ is the iteration counter and $(\sigma_{k})_{k\geq1}$ is a well-chosen regularization sequence. Over the past years, a significant list of publications appeared, proving rates of convergence with respect to the inner (lower-level) and the outer (upper-level) problem in the setting of hierarchical optimization as in Section \ref{ex:SBP} \cite{beck2014first,SabSht17,merchav2023convex}.
Building on the iterative regularization structure, adaptive versions of a related proximal-gradient strategy are studied in \cite{latafat2025convergence} without rates on the upper level problem. \cite{shtern2025convergence} and \cite{boct2025accelerating} present rates for both upper and lower level problems for the iterative proximal gradient method and its accelerated version, which was first analyzed in \cite{merchav2024fast}. Beyond proximal-gradient based methods, projection-free schemes are developed in \cite{doron2023methodology,giang2024projection}. Moving beyond optimization, there exists a significant body of literature, rooted in inverse problems and signal processing, which studied solution methods for selecting a solution to a variational inequality given a pre-defined selection criterion \cite{ogura2002non,xu2003convergence,marino2011explicit,ono2014hierarchical}. 
All these papers focus on asymptotic convergence of the sequence generated by tailor-made numerical algorithms. The only results on iteration complexity we are aware of are published in the papers \cite{kaushik2023incremental,Kaushik:2021aa,jalilzadeh2024stochastic}. Moving beyond the function-case in the upper-level (equilibrium selection), there exist a few papers on numerical methods for solving hierarchical VIs \cite{van2021regularization,ThongNumerical20,lampariello2022solution,LamSagSIOPT25,samadi_improved_2025,alves2025inertial,cortild2025regularization}. Among those, only \cite{lampariello2022solution,LamSagSIOPT25,samadi_improved_2025,alves2025inertial,cortild2025regularization} contain complexity statements.

In \cite{LamSagSIOPT25}, the authors consider hierarchical games with non-smooth payoffs, which is a different setting from ours since in their case $g_1=g_2\equiv 0$ and $\opF_1,\opF_2$ are given by subdifferentials of convex functions and may be set-valued. 
The recent papers \cite{lampariello2022solution,samadi_improved_2025,alves2025inertial} study the finite-dimensional hierarchical variational inequality problem, obtained from our model template \eqref{eq:P} by imposing $g_{1}=0$ and $g_{2}$ to be an indicator function over a \emph{compact} convex set. \cite{lampariello2022solution} use a projected-gradient-type algorithm and obtain complexity $O(\varepsilon^{-4})$ to reduce the accuracy of approximate solution below $\varepsilon$ on both upper and lower level.
\cite{samadi_improved_2025,alves2025inertial} use the Korpelevich extragradient method, and the obtained rates are $O(1/k^{\delta})$ (lower level) and $O(1/k^{1-\delta})$ (upper level), where $\delta \in (0,1)$ is a parameter.

\subsection{Contributions}
Our aim with this paper is to settle the complexity issue for a very general family of nested equilibrium problems formulated in the framework \eqref{eq:P}, and thereby extend all the existing results on hierarchical monotone variational inequalities. Our main contributions are as follows:
\paragraph{More general problem formulation:} Unlike \cite{lampariello2022solution,samadi_improved_2025,alves2025inertial}, the hierarchical equilibrium model \eqref{eq:P} contains general convex composite terms $g_1,g_2$. Specifically, previous papers treat only the special case $g_{1}=0$ and $g_{2}=\indicator_{\scrC}$, the indicator function of a \textit{compact} convex and nonempty set $\scrC$ in a finite-dimensional Euclidean vector space. Additionally, we consider the infinite-dimensional real Hilbert space setting, and general $\HVI$'s. Besides the gain in modeling, adding these convex and potentially non-smooth functions is an important extension, as it allows treating the non-smooth composite convex model when specialized to the potential case \cite{merchav2024fast}. 

\paragraph{Enhanced Algorithmic design and combined rate statements:}  \cite{samadi_improved_2025,alves2025inertial} approach the numerical resolution of problem \eqref{eq:P} via the Korpelevich double-call version of the extragradient method. Our approach instead builds on the Popov single-call version extragradient method \cite{hsieh2019convergence}. This modification allows us to save one evaluation of the operators per iteration, which can be a significant computational advantage in cost-sensitive applications. Within this algorithmic setting, we derive sublinear convergence rates for the lower- and upper-level problem and improved rates under the strong monotonicity of $\opF_1$. Specifically, using a polynomial regularization sequence $\sigma_{k}=O(k^{-\delta})$ for $\delta\in(0,1)$, we demonstrate upper convergence rate bounds on the order of $O(k^{-\delta})$ for the gap function associated with the lower-level equilibrium problem (\emph{feasibility gap}), and a $O(k^{-(1-\delta)})$ convergence rate bound for the gap function of the entire hierarchical problem (\emph{optimality gap}). Furthermore, under the additional assumption of strong monotonicity of $\opF_1$, we obtain improved rates $\widetilde{O}(k^{-1})$ for both levels. These bounds match the existing bounds reported in \cite{samadi_improved_2025,alves2025inertial} and improve upon the bounds in \cite{lampariello2022solution,LamSagSIOPT25}, while applying to a much broader class of equilibrium problems and requiring fewer operator evaluations. 
 
 \cite{cortild2025regularization} is very much related in terms of the problem formulation. They, focus on the smooth VI case in the upper level, and an abstract monotone inclusion in the lower level. Hence, the problem formulation is not entirely comparable to our HVI framework (the upper level is less general, but the lower level is more general). To solve this hierarchical equilibrium problem, they develop a double-loop diagonal scheme inspired by \cite{Facchinei:2014aa}. Under bounded domain assumption, the proof complexity statements in terms of similar performance measures as we employ in this paper. Specifically, they establish an optimality gap rate of $\tilde{O}(k^{-\delta})$ and a feasibility gap rate on the order of $\tilde{O}(k^{-\min(\delta,1-\delta)})$ for $\delta\in(0,1)$ measuring the strength of the Tikhonov regularization. The notation $\tilde{O}(\cdot)$ hides logarithmic factors, contaminating their complexity due to the double loop architecture. 

\paragraph{Unbounded domains:} An important assumption in the complexity analysis performed in \cite{lampariello2022solution,samadi_improved_2025,alves2025inertial} is compactness of the domains over which the variational inequalities are solved. We abandon this assumption. Instead, our proof builds on ideas originating from the asymptotic analysis of non-autonomous evolution equations due to \cite{attouch2018asymptotic}. In the potential case, these ideas have already been successfully applied to the study of hierarchical optimization problems \cite{Cabot2005}. In particular, \cite{boct2025accelerating} make the role of this technique very transparent by emphasizing the importance of the geometry of the lower level problem in establishing a unified complexity result at both levels simultaneously. A methodological contribution of this work is to show how these ideas naturally translate to the HVI setting. Our proof method is quite flexible and can be fruitfully applied to other settings as well, as demonstrated in \cite{dvurechensky2026stochastic}. We therefore belief that this methodological contribution will be of some interest to the community. 

\subsection{Outline}
The outline of this paper is as follows. Section \ref{sec:VI} collects relevant facts on VIs, defining the geometric properties we assume for the solution set of the lower-level problem, and introducing the merit functions we employ in our main results. Section \ref{sec:OMD} contains the main algorithmic paradigm we employ to prove our rates. Section \ref{sec:AllProofs} details the proof of our two main theorems. Section \ref{sec:extensions} reports important extensions. We establish improved rates under strong monotonicity of the upper-level $\HVI$ and show that the main statements extend to optimistic versions of the forward-backward-forward splitting, enjoying only one proximal operator evaluation per iteration. Section \ref{sec:numerics} presents numerical experiments demonstrating our results.

\section{Facts on variational inequalities}
\label{sec:VI}
%
\subsection{Preliminaries}
We work in the real Hilbert space $\scrZ$ with scalar product $\inner{\cdot,\cdot}$ and induced norm $\norm{\cdot}$. For $x \in \scrZ$ and $r>0$, $\ball(x,r)$ denotes the ball of radius $r$ around $x$ defined by the norm in $\scrZ$. The domain of a function $G:\scrZ\to\R\cup\{+\infty\}\eqdef(-\infty,\infty]$ is defined as $\scrD_G\eqdef\dom(G)\eqdef\{z\in\scrZ\vert G(z)<+\infty\}$. The class of proper, convex, and lower semi-continuous functions $g:\scrZ\to(-\infty,\infty]$ is denoted by $\Gamma_{0}(\scrZ)$. The proximal operator of $G$ is defined as $\prox_{G}(u)\eqdef\argmin_{z\in\scrZ}\{G(z)+\frac{1}{2}\norm{z-u}^{2}\}.$ If $x,u\in\scrZ$, we have
\begin{equation}\label{eq:prox}
x=\prox_{G}(u)\iff \inner{u-x,z-x}\leq G(z)-G(x)\qquad \forall z\in\scrZ.
\end{equation}
For a closed convex set $\scrC\subset\scrZ$, the metric projection is $\Pi_{\scrC}(x)\eqdef\prox_{\indicator_{\scrC}}(x)$, where the convex indicator function $\indicator_{\scrC}:\scrZ\to(-\infty,\infty]$ is defined by $\indicator_{\scrC}(x)=0$ if $x\in\scrC$, and $\indicator_{\scrC}(x)=+\infty$ else. We also denote  $\dist(x,\scrC)\eqdef\norm{x-\Pi_{\scrC}(x)}$. 
The polar cone attached to a set $\scrC\subset\scrZ$ 
is $\scrC^{\circ}\eqdef\{z\in\scrZ\vert \inner{z,x}\leq 0 \quad \forall x\in\scrC\}$. 
The normal cone at point $z\in \scrC$ is defined as $\NC_{\scrC}(z)\eqdef\{p\in\scrZ\vert \inner{p,z'-z}\leq 0\quad\forall z'\in\scrC\}$. 
For any convex $\scrC\subset\scrZ$ (not neccesarily closed),
the tangent cone $\TC_{\scrC}(x)$ of $\scrC$ at $x\in\scrC$ is polar to the normal cone, and we have the explicit expression $\TC_{\scrC}(x)=\cl\left(\bigcup_{\lambda>0}\frac{\scrC-x}{\lambda}\right)$.

The Fenchel conjugate of $f\in\Gamma_{0}(\scrZ)$ is $f^{\ast}(y)\eqdef\sup_{x\in\dom(f)}\{\inner{y,x}-f(x)\}.$ 
Specifically,  given a closed convex set $\scrC\subset\scrZ$, the Fenchel conjugate of $\indicator_\scrC$ is 
the support function of set $\scrC$ defined as $\supp(p|\scrC)\eqdef\sup_{z\in\scrC}\inner{p,z}.$ Note that 
\begin{equation}\label{eq:support_equality}
\inner{p,z}=\supp(p|\scrC)\qquad \forall p\in\NC_{\scrC}(z).
\end{equation}
A mapping (operator) $\opF:\scrZ\to\scrZ$ is $\mu$-strongly-monotone ($\mu\geq0$) if 
\begin{equation*}
\inner{\opF(z_{1})-\opF(z_{2}),z_{1}-z_{2}}\geq \mu\norm{z_{1}-z_{2}}^{2}\qquad \forall z_{1},z_{2}\in\scrZ.
\end{equation*}
A $0$-monotone operator is simply called monotone. $\opF$ is $L$-Lipschitz continuous if 
\begin{equation*}
\norm{\opF(z_{1})-\opF(z_{2})}\leq L\norm{z_{1}-z_{2}}\qquad\forall z_{1},z_{2}\in\scrZ.
\end{equation*}

\subsection{Hemi-variational inequalities (HVI's)}
Given a monotone and Lipschitz continuous mapping $\opF:\scrZ\to\scrZ$ and a function $g\in\Gamma_{0}(\scrZ)$, the $\HVI$ problem $\HVI(\opF,g)$ is to find $z^{\ast}\in\scrZ$ such that 
\begin{equation}\label{eq:VI_general}
\inner{\opF(z^{\ast}),u-z^{\ast}}+g(u)-g(z^{\ast})\geq 0\qquad\forall u\in\scrZ.
\end{equation}
One can rewrite \eqref{eq:VI_general} as a monotone inclusion $0\in(\opF+\partial g)(z^{\ast})$, or equivalently the solution set is $\scrS=\zer(\opF+\partial g)$. If $g=\indicator_{\scrC}$ for a closed convex set $\scrC\subset\scrZ$, we recover the variational inequality problem $\VI(\opF,\scrC)$:
\begin{equation*}
\text{Find }z^{\ast}\in\scrC\text{ s.t. } \inner{\opF(z^{\ast}),u-z^{\ast}}\geq 0 \qquad\forall u\in\scrC. 
\end{equation*}
Let $\scrC\subset\dom(g)$ be a given compact subset. A popular merit function for $\HVI(\opF,g)$ is the localized (dual) gap function 
\begin{equation}
\Theta(z\vert\opF,g,\scrC)\eqdef \sup_{y\in\scrC}\{\inner{\opF(y),z-y}+g(z)-g(y)\}=\sup_{y\in\scrC} H^{(\opF,g)}(z,y),
\label{eq:gap_Def_new}
\end{equation}
where $H^{(\opF,g)}:\scrZ\times\scrZ\to[-\infty,\infty]$ is defined as 
\begin{equation}
    \label{eq:HFG_def}
    H^{(\opF,g)}(z,y)\eqdef \inner{\opF(y),z-y}+g(z)-g(y).
\end{equation}
Historically, this localized version of a gap function can be traced back to \cite{Minty62,kinderlehrer2000introduction}. Lemma \ref{lem:gap} below, whose proof is given in Appendix~\ref{app:proofs}, extends a result proved in \cite{NesDual07}, and summarizes its role for the numerical resolution of the problem $\HVI(\opF,g)$ under continuity assumptions on $\opF$.
\begin{lemma}\label{lem:gap}
    Let $\scrC\subset\dom(g)$  be a nonempty compact convex set. Consider problem $\HVI(\opF,g)$ with $\opF:\scrZ\to\scrZ$ monotone and Lipschitz continuous. The function $x\mapsto \Theta(x\vert\opF,g,\scrC)$ is well-defined and convex on $\scrZ$. For any $x\in\scrC$ we have $\Theta(x\vert \opF,g,\scrC)\geq 0.$ Moreover, if $x\in\scrC$ is a solution to \eqref{eq:VI_general}, then $\Theta(x\vert\opF,g,\scrC)=0$. Conversely, if $\Theta(x\vert\opF,g,\scrC)=0$ for some $x\in\scrC$ for which there exists an $\eps>0$ such that $\ball(x,\eps)\cap\scrC=\ball(x,\eps)\cap\dom(g)$, then $x$ is a solution of \eqref{eq:VI_general}.   
\end{lemma}
The above Lemma, in particular, means that under mild assumptions on the set $\scrC$, bounding the gap relative to the set $\scrC$ is equivalent to solving the corresponding HVI. Thus, our main algorithmic focus in what follows is on finding approximate solutions with small gap relative to an arbitrary compact set $\scrC$.
 
\subsection{Sharpness and error-bound property}
Our geometric framework is phrased in terms of weak sharpness of the lower-level solution set $\scrS_{2}$. Originally formulated for optimization problems in \cite{burke1993weak}, weak sharpness in the context of variational inequalities has been defined in \cite{patriksson1993unified}. Subsequently, implications in terms of error bounds of primal and dual gap functions have been developed in \cite{marcotte1998weak,liu2016characterization}. The following definition of weak sharpness is from \cite{Huang:2018aa}. 

\begin{definition}
Let $\opF:\scrZ\to\scrZ$ be continuous and monotone over $\dom(g)\subset\scrZ$. Assume $\dom(g)$ is closed. The solution set $\scrS$ of $\HVI(\opF,g)$ is \emph{weakly sharp} if there exists $\tau>0$ such that 
\[
(\forall z^{\ast}\in\scrS):\quad \tau\ball(0,1)\subseteq\opF(z^{\ast})+\partial g(z^{\ast})+[\TC_{\dom(g)}(z^{\ast})\cap\NC_{\scrS}(z^{\ast})]^{\circ} .
\]
\end{definition}

Based on \cite{Huang:2018aa}, we give the following characterization of weak sharpness in terms of an error bound for the dual gap function of $\HVI(\opF,g)$;
see the proof in Appendix \ref{app:proofs}.
\begin{proposition}\label{prop:EB1}
Consider problem $\HVI(\opF,g)$ with $\dom(g)$ a closed, convex, and nonempty subset of $\scrZ$. If the solution set $\scrS$ of $\HVI(\opF,g)$ is weakly sharp, then 
\begin{equation*}
(\forall z\in \dom(g)):\quad \Theta(z\vert\opF,g,\dom(g))\geq \tau\dist(z,\scrS).
\end{equation*}
\end{proposition}

Proposition~\ref{prop:EB1} shows that weak sharpness implies that the gap function satisfies an error bound property \cite{pang1997error},
motivating the following definition:
\begin{definition}[Weak Sharpness]
\label{def:WS}
Let $\scrS$ be the solution set of $\HVI(\opF,g)$. We say $\scrS$ is $(\alpha,\rho)$-weakly sharp with $\alpha>0$ and $\rho\geq 1$ if
\begin{equation}\label{eq:WSDef}
(\forall z^{\ast}\in\scrS)(\forall z\in \dom(g)):\quad H^{(\opF,g)}(z,z^{\ast})\geq\frac{\alpha}{\rho}\dist(z,\scrS)^{\rho}.
\end{equation}
\end{definition}
\begin{remark}
     An important class of examples arises when $\rho=1$, notably in monotone linear complementarity problems in finite dimensions under nondegeneracy conditions \cite{burke1993weak,pang1997error}. In the case where the $\HVI(\opF,g)$ reduces to a convex optimization problem, weak sharpness implies an H\"olderian error bound, an assumption already imposed by \cite{Cabot2005} in the context of hierarchical minimization. Specifically, let us assume that $\opF=0$ and $g\in\Gamma_{0}(\scrZ)$. Then $\scrS=\argmin g$, and accordingly, $H^{(0,g)}(z,z^{*})=g(z)-\min g$ for $z^{*}\in\scrS$. Hence, \eqref{eq:WSDef} implies 
    \[
    \frac{\alpha}{\rho}\dist(u,\argmin g)^{\rho}\leq g(u)-\min g\qquad \forall u\in\scrZ.
    \]
If $\rho=1$, this is the weak-sharpness condition of \cite{burke1993weak}. The case $\rho=2$ corresponds to the "quadratic growth" condition of \cite{drusvyatskiy2018error}, a relaxation of strong convexity.\close
\end{remark}

\subsection{Constraint qualifications}
The following set of assumptions shall be in place throughout the paper. 
\begin{assumption}\label{ass:Mappings}
For $i=1,2$ the following properties do hold:
\begin{itemize}
\item[(i)] $\opF_{i}:\scrZ\to \scrZ$ is $L_{\opF_i}$-Lipschitz continuous and monotone; 
\item[(ii)] $g_{i}\in\Gamma_{0}(\scrZ)$ and $\dom(g_{1})\cap\dom(g_{2})\neq \emptyset$;
\item[(iii)] $\scrS_{2}\eqdef \zer(\opF_{2}+\partial g_{2})\neq\emptyset$.
\end{itemize}
\end{assumption}
We remark that Assumption \ref{ass:Mappings} implies that $\scrS_{2}$ is a non-empty, closed, and convex set. The next assumption is essentially a constraint qualification condition on the composite function 
$
x\mapsto \inner{\opF_{1}(\bar{x}),x-\bar{x}}+g_{1}(x)-g_{1}(\bar{x})+\indicator_{\scrS_{2}}(x). 
$
\begin{assumption}
\label{ass:CQ}
The solution set $\scrS_{1}$ of \eqref{eq:P} satisfies
$\scrS_{1}=\zer(\opF_{1}+\partial g_{1}+\NC_{\scrS_{2}})\neq\emptyset$.
 
\end{assumption}
Similar assumptions have been made in \cite{Bot:2014aa} for solving constrained VI's. Assumption \ref{ass:CQ} is motivated by conditions ensuring a non-smooth sum rule of the form
\begin{equation*}
\partial_{x}\left(\inner{\opF_{1}(\bar{x}),\Id_{\scrZ}(\bullet)-\bar{x}}+g_{1}(\bullet)-g_{1}(\bar{x})+\indicator_{\scrS_{2}}(\bullet)\right)(\bar{x})=\opF_{1}(\bar{x})+\partial g_{1}(\bar{x})+\NC_{\scrS_{2}}(\bar{x}). 
\end{equation*}
Common assumptions ensuring this property can be found in \cite[Corollary 16.48]{BauCom16}.

\subsection{Gap functions for hierarchical HVI's}
\label{sec:gapBilevel}
Since we are searching for a solution of a hierarchical system of VIs, we have to introduce different merit functions for measuring the feasibility and the optimality of a test point $z\in\scrZ$.
\begin{definition}[Feasibility gap]
We define the \emph{feasibility gap} over a compact nonempty set $\scrU\subseteq\dom(g_{2})$ as
\begin{equation}\label{eq:FeasGap}
\Theta_{\rm Feas}(z\vert\scrU)\eqdef \Theta(z\vert \opF_{2},g_{2},\scrU)=\sup_{y\in\scrU}\{\inner{\opF_{2}(y),z-y}+g_{2}(z)-g_{2}(y)\}.
\end{equation}
\end{definition}
\begin{definition}[Optimality gap]
The \emph{optimality gap} over a compact set $\scrU\subseteq\dom(g_{1})$ with $\scrU\cap\scrS_{2}\neq\emptyset$ is defined as 
\begin{equation}\label{eq:OptGap}
\Theta_{\rm Opt}(z\vert\scrU\cap\scrS_{2})\eqdef \Theta(z\vert \opF_{1},g_{1},\scrU\cap\scrS_{2})=\sup_{y\in\scrU\cap\scrS_{2}}\{\inner{\opF_{1}(y),z-y}+g_{1}(z)-g_{1}(y)\}.
\end{equation}
\end{definition}
We note that if $z\in\scrU\cap\scrS_{2}$, and the regularity conditions stated in Lemma \ref{lem:gap} hold, then by the same lemma, the inequality $\Theta_{\rm Opt}(z\vert\scrU\cap\scrS_{2})\leq 0$ is equivalent to $z\in \scrS_{1}$. However, in general an inequality of the form $\Theta_{\rm Opt}(z\vert\scrU\cap\scrS_{2})\leq 0$, does not tell us much about the qualitative properties of the test point if $z\notin\scrS_{2}$. It is thus important to obtain a lower bound on the optimality gap. This is achieved by Lemma \ref{lem:LB} below, with proof provided in Appendix \ref{app:proofs}. It extends the corresponding result in \cite{boct2025accelerating} to the operator case. \cite{samadi_improved_2025} and \cite{alves2025inertial} prove this result for the special case of bilevel VIs on compact domains.

\begin{lemma}\label{lem:LB}
Consider problem \eqref{eq:P}. Let Assumption \ref{ass:Mappings} and \ref{ass:CQ} hold. Let $\scrU_1\subseteq\dom(g_{1})$ be a nonempty compact set with $\scrU_1\cap\scrS_{1}\neq\emptyset$. Then, there exists a constant $B_{\scrU_1}>0$ such that
\begin{equation}\label{eq:LB1}
\Theta_{\rm Opt}(z\vert\scrU_1\cap\scrS_1)\geq-B_{\scrU_1}\dist(z,\scrS_{2}), \quad\forall z\in\scrZ.
\end{equation}
Suppose $\scrS_{2}$ is $(\alpha,\rho)$-weakly sharp. Then for all nonempty and compact subsets $\scrU_2\subseteq\dom(g_{2})$ satisfying $\scrU_2\cap \scrS_2\ne \emptyset$, and all $z\in\dom(g_{2})$, we have  
\begin{equation}\label{eq:WS}
\dist(z,\scrS_{2})\leq\left[\frac{\rho}{\alpha}\Theta_{\rm Feas}(z\vert \scrU_2\cap \scrS_2)\right]^{1/\rho}.
\end{equation}
\end{lemma}

\section{Optimistic extragradient method}
\label{sec:OMD}
%

Our algorithmic design follows the popular Tikhonov regularization approach and, at iteration $k\geq 1$, uses the operators $\opV_k$ and $G_k$ defined in \eqref{eq:VkGk_def}, which combine the upper- and lower-level data. A decreasing sequence of regularization parameters $\sigma_k \geq 0 $ determines the relative importance of the upper-level problem relative to the lower-level problem. Being a sum of monotone and Lipschitz continuous mappings, $\opV_{k}$ is monotone and $L_k$-Lipschitz with $L_{k}\eqdef L_{\opF_{2}}+\sigma_{k}L_{\opF_{1}}$. 
The conceptual algorithm we propose for solving \eqref{eq:P} is described in Algorithm \ref{alg:Alg1}. Once concrete definitions for the step-sizes $(t_{k})_{k\geq1}$ and the regularization parameters $(\sigma_{k})_{k\geq1}$ are given, we obtain a bona-fide algorithm. Importantly, it should be observed that our method only requires one evaluation of operators $\opF_1$, $\opF_2$ per iteration, in contrast to previous works \cite{samadi_improved_2025,alves2025inertial}.
\begin{algorithm}
\caption{Optimistic extragradient method for hierarchical HVI's \eqref{eq:P}}\label{alg:Alg1}
\begin{algorithmic}
\Require $z^{1}=z^{1/2}$, step-size sequence $(t_{k})_{k\geq 1} \subset \R_{>0}$, decreasing regularization sequence $(\sigma_{k})_{k\geq 1}\subset \R_{>0}$ satisfying \vspace{-6pt}
\begin{equation}\label{eq:tsigma}
\begin{split}
&\lim_{k\to\infty}\sigma_{k}=0,\quad \lim_{K\to\infty}\sum_{k=1}^{K}t_{k}= \infty,\\
&\lim_{K\to\infty} \sigma_{T}\sum_{k=1}^{K}t_{k}= \infty \text{, and }\lim_{K\to\infty}\frac{\sum_{k=1}^{K}t_{k}\sigma_{k}}{\sum_{k=1}^{K}t_{k}}=0.
\end{split}
\end{equation}
\vspace{-20pt}
\State Set $k=1$ 
\While{stopping criterion not met}
\State Compute \vspace{-10pt}
\begin{align*}
&z^{k+1/2}=\prox_{t_{k}G_{k}}(z^{k}-t_{k}(\opF_{2}(z^{k-1/2})+\sigma_{k}\opF_{1}(z^{k-1/2}))),\\ 
&z^{k+1}=\prox_{t_{k}G_{k}}(z^{k}-t_{k}(\opF_{2}(z^{k+1/2})+\sigma_{k}\opF_{1}(z^{k+1/2})).
\end{align*}
\vskip-10pt
\EndWhile
\end{algorithmic}
\end{algorithm}

\begin{remark}
To simplify the presentation, we implicitly assume that the proximal mapping of $G_{k}$ can be evaluated efficiently. In many practically relevant situations this is indeed the case. Trivially, if $g_{1}$ or $g_{2}$ is zero, we obtain an easier implementation. Other examples include the cases where either $g_{1}$ is an $L_{1}$ norm and $g_{2}$ is an indicator function of a simple set, such as a box or an $L_{2}$-ball \cite{BauCom16}. In general though, the evaluation of the proximal operator of the composite mapping $G_{k}=\sigma_{k}g_{1}+g_{2}$ is not an easy task. We can circumvent this problem, by adopting a lifted formulation  \cite{doron2023methodology}. To obtain such a formulation, let $x=[z_{1},z_{2}]\in\scrZ\otimes\scrZ\eqdef \scrX$ and consider the operators 
\begin{align*}
&\opA_{1}(x)=\begin{pmatrix} \partial g_{1}(z_{1}) \\ 0 \end{pmatrix}, \opA_{2}(x)=\begin{pmatrix} 0 \\ \partial g_{2}(z_{2}) \end{pmatrix}:\scrX\to 2^{\scrX},\text{ and } \\ 
&\opB_{1}(x)=\begin{pmatrix} \opF_{1}(z_{1})\\ 0 \end{pmatrix}, \opB_{2}(x)=\begin{pmatrix} z_{1}-z_{2} \\ \opF_{2}(z_{2})-z_{1}+z_{2}\end{pmatrix} :\scrX\to\scrX.
\end{align*}
It is easy to see that 
$
\tilde{\scrS}_{2}=\zer(\opA_{2}+\opB_{2})=\{(z_{1},z_{2})\in\scrZ\times\scrZ\vert z_{1}=z_{2},\text{ and } z_{2}\in \zer(\opF_{2}+\partial g_{2})\}. 
$
This means we can obtain $\scrS_{2}$ as the projection of points in $\tilde{\scrS}_{2}$ onto its first coordinate. Moreover, the problem 
\begin{equation*}
\text{Find } \bar{x}=[\bar{z}_{1},\bar{z}_{2}] \in\scrX \text{ s.t. }\inner{\opF_{1}(\bar{z}_{1}),z_{1}-\bar{z}_{1}}+g_{1}(z_{1})-g_{1}(\bar{z}_{1})\geq 0 \quad\forall  \bar{x}\in\tilde{\scrS}_{2} 
\end{equation*}
shares the same solution as the non-lifted formulation \eqref{eq:P} after projection. We can thus approach the solution of the lifted problem by performing the parallel computations 
\begin{align*}
x^{k+1/2}=(\Id_{\scrX}+t_{k}\scrA_{k})^{-1}(x^{k}-t_{k}\scrB_{k}(x^{k-1/2})),\\ 
x^{k+1}=(\Id_{\scrX}+t_{k}\scrA_{k})^{-1}(x^{k}-t_{k}\scrB_{k}(x^{k+1/2})),
\end{align*}
where $\scrA_{k}\eqdef\sigma_{k}\opA_{1}+\opA_{2}$ and $\scrB_{k}\eqdef \sigma_{k}\opB_{1}+\opB_{2}$. Thanks to the direct product structure of the lifting, it is easy to see that the resolvents indeed involve parallel computations on the individual factors. As an illustration, the update of $x^{k+1/2}=[z_{1}^{k+1/2},z_{2}^{k+1/2}]$ reads explicitly as 
\begin{align*}
z_{1}^{k+1/2}&=\prox_{t_{k}\sigma_{k}g_{1}}(z_{1}^{k}-t_{k}(z_{1}^{k-1/2}-z_{2}^{k-1/2}+\sigma_{k}\opF_{1}(z_{1}^{k-1/2}))),\\
z^{k+1/2}_{2}&=\prox_{t_{k}g_{2}}(z_{2}^{k}-t_{k}(z_{2}^{k-1/2}-z_{1}^{k-1/2}+\opF_{2}(z_{2}^{k-1/2})).
\end{align*}
\close
\end{remark}

\subsection{Statement of the main results}
\label{sec:mainresults}
The main results of this paper are two convergence rate statements for the averaged iterates generated by Algorithm \ref{alg:Alg1} in terms of the gap functions   \eqref{eq:FeasGap} and \eqref{eq:OptGap}. We work in a specific geometric setting in which we assume some structure on the solution set of the lower level problem $\HVI(\opF_{2},g_{2})$. Associated to the bifunction $H^{(\opF_{2},g_{2})}$, Appendix \ref{app:FP} introduces the mapping 
\begin{equation}\label{eq:phimaintext}
\varphi^{(\opF_{2},g_{2})}(z,u)\eqdef \sup_{y\in\dom(g_{2})} \{H^{(\opF_{2},g_{2})}(z,y)+\inner{y,u}\}.
\end{equation}
This function encodes dual properties of $\HVI(\opF_{2},g_{2})$, since it holds that (cf. eq. \eqref{eq:dualphi}) 
\begin{equation}\label{eq:dualphimaintext}
\varphi^{(\opF_{2},g_{2})}(z,u)\leq\sup_{y\in\dom(g_{2})}\{\inner{y,u}-H^{(\opF_{2},g_{2})}(y,z)\}=\left(H^{(\opF_{2},g_{2})}(\bullet,z)\right)^{\ast}(u). 
    \end{equation}
 The following summability condition is essentially due to \cite{Attouch:2010aa,AttCzarPey11}:
\begin{assumption}\label{ass:ACcondition}[Attouch-Czarnecki condition]
The step size sequence $(t_{k})_{k\geq 1}$ and the regularization sequence $(\sigma_{k})_{k\geq 1}$ satisfy
\begin{equation}\label{eq:AttCza}
(\forall p\in\text{Range}(\NC_{\scrS_{2}})):\;\sum_{k=1}^{\infty}t_{k}\left[\sup_{z\in\scrS_{2}}\varphi^{(\opF_{2},g_{2})}(z,\sigma_{k}p)-\supp(\sigma_{k}p\vert\scrS_{2})\right]<\infty.
\end{equation}
\end{assumption}
    To understand the meaning of Assumption \ref{ass:ACcondition}, it is instructive to specialize our setting to the simple convex bilevel optimization case (cf. Section \ref{ex:SBP}). In that case, we have $\opF_{2}=\nabla f_{2}$ and $\scrS_{2}=\argmin_{z}(f_{2}+g_{2})(z)$, and thus 
    \begin{align*}
       H^{(\opF_{2},g_{2})}(y,z)&\leq f_{2}(y)-f_{2}(z)+g_{2}(y)-g_{2}(z) \eqdef \hat{f}_{2}(y)-\hat{f}_{2}(z)\\
       &=(\hat{f}_{2}-\min\hat{f}_{2})(y)-(\hat{f}_{2}-\min\hat{f}_{2})(z).
    \end{align*}
Hence, $\varphi^{(\opF_{2},g_{2})}(z,p)\leq (\hat{f}_{2}-\min\hat{f}_{2})(z)+(\hat{f}_{2}-\min\hat{f}_{2})^{*}(p)$. Since for $z\in\scrS_{2}$ it holds $(\hat{f}_{2}-\min\hat{f}_{2})(z)=0$, this further implies 
\begin{equation*}
\sup_{z\in\scrS_{2}}\varphi^{(\opF_{2},g_{2})}(z,\sigma_{k}p)-\supp(\sigma_{k}p\vert\scrS_{2})\leq (\hat{f}_{2}-\min\hat{f}_{2})^{*}(\sigma_{k}p)-\supp(\sigma_{k}p\vert\scrS_{2}). 
\end{equation*}
In \cite{AttCzarPey11,Peypouquet:2012aa,boct2025accelerating} the following summability condition is imposed
\begin{equation*}
\sum_{k=1}^{\infty}t_{k}\left[(\hat{f}_{2}-\min\hat{f}_{2})^{*}(\sigma_{k}p)-\supp(\sigma_{k}p\vert\scrS_{2})\right]<\infty.
\end{equation*}
We see that this condition is more restrictive than our condition \eqref{eq:AttCza}.

\begin{remark}
\label{rem:Tikhonov}
Assumption \ref{ass:ACcondition} looks quite abstract and difficult to verify. However, as already observed in \cite{boct2025accelerating}, it fits very natural to the geometric setting of this paper. Indeed, let us assume that $\scrS_{2}$ is $(\alpha,\rho)$-weakly sharp, with $\alpha>0$ and $\rho>1$. According to Definition \ref{def:WS}, for all $z\in\dom(g_{2})$ and for all $z^{*}\in\scrS_{2}$, we have $H^{(\opF_{2},g_{2})}(z,z^{*})\geq \alpha\rho^{-1}\dist(z,\scrS_{2})^{\rho}$. Hence, eq. \eqref{eq:dualphimaintext} yields for all $z\in\scrS_{2}$
\begin{align*}
\varphi^{(\opF_{2},g_{2})}(z,\sigma_{k}p^{\ast})-\supp(\sigma_{k}p^{\ast}\vert\scrS_{2})&\leq \left(H^{(\opF_{2},g_{2})}(\bullet,z)\right)^{\ast}(\sigma_{k}p^{\ast})-\supp(\sigma_{k}p^{\ast}\vert\scrS_{2})\\
&\leq \alpha^{-\frac{1}{\rho-1}}\left(\frac{\rho-1}{\rho}\right)\sigma_{k}^{\frac{\rho}{\rho-1}}\norm{p^{\ast}}^{\frac{\rho}{\rho-1}}. 
\end{align*}
This shows that under the $(\alpha,\rho)$-weak sharpness condition, the summability condition \eqref{eq:AttCza} is satisfied whenever $\sum_{k\geq 1}t_{k}\sigma^{\frac{\rho}{\rho-1}}_{k}<\infty.$ 
\close
\end{remark}
We additionally assume that the non-smooth part $g_{1}$ of $\HVI(\opF_{1},g_{1})$ to satisfies a mild finite variation property. Specifically, we define the variation over a set $\scrA\times\scrB\subseteq\scrZ\times\scrZ$ as
\begin{equation*}
\Var(g_{1}\vert \scrA\times\scrB)\eqdef \sup_{(x,y)\in\scrA\times\scrB}\abs{g_{1}(x)-g_{1}(y)}.
\end{equation*} 

\begin{assumption}
\label{ass:BVg}
The set $\dom(g_1)\cap \dom(g_2)$ is closed, and for any compact sets $\scrA,\scrB\subset\scrZ$ satisfying $\scrA,\scrB \subset\dom(g_{1})\cap\dom(g_{2})$, 
we have $\Var(g_{1}\vert \scrA\times\scrB)<\infty$.
\end{assumption}
Assumption \ref{ass:ACcondition} will allow us to prove that the sequences generated by Algorithm \ref{alg:Alg1} are bounded, an important step towards deriving convergence rates in terms of the averaged trajectory 
\begin{equation}\label{eq:ergodic}
\bar{z}^{K}\eqdef\frac{1}{T_{K}}\sum_{k=1}^{K}t_{k}z^{k+1/2}, \quad T_{K}\eqdef\sum_{k=1}^{K}t_{k}.
\end{equation}

\begin{theorem}\label{th:Gap1}
Consider problem \eqref{eq:P}. Let Assumptions \ref{ass:Mappings}, \ref{ass:CQ}, and \ref{ass:BVg} hold. Assume additionally that either Assumption \ref{ass:ACcondition} holds, or that  $\dom(g_{1})\cap\dom(g_{2})$ is compact. 
Let $(\bar{z}^{k})_{k\geq1}$ as in \eqref{eq:ergodic} be generated by Algorithm \ref{alg:Alg1} with $16t_{k}^{2}L_{k}^2\leq 1$ for all $k\geq 1$.
\begin{itemize}
\item[(i)] 
Let $\scrU_{2}\subset\dom(g_{1})\cap\dom(g_{2})$ be non empty and compact. Then, there exists a constant $C_{\scrU_{2}}>0$ for which
\begin{equation}
    \label{eq:Feas_unbounded}
    \Theta_{\rm Feas}(\bar{z}^{K}\vert\scrU_2)\leq \frac{\sup_{z\in\scrU_2}\norm{z^{1}-z}^{2}}{2T_{K}}+ \frac{\sum_{k=1}^{K}t_k \sigma_k}{T_{K}} C_{\scrU_2}.
\end{equation}
\item[(ii)] Let $\scrU_{1}\subset\dom(g_{1})\cap\dom(g_{2})$ be a compact subset with $\scrU_{1}\cap\scrS_{1}\neq\emptyset$. Then there exist constants $B_{\scrU_{1}},C_{\scrU_{1}}>0$ such that 
\begin{equation}\label{eq:Opt_unbounded}
-B_{\scrU_{1}}\dist(\bar{z}^{K},\scrS_{2})\leq \Theta_{\rm Opt}(\bar{z}^{K}\vert\scrU_{1}\cap\scrS_{2})\leq \frac{C_{\scrU_{1}}}{T_{K}\sigma_{K}}.
\end{equation}
\item[(iii)] If the lower level solution set $\scrS_{2}$ is $(\alpha,\rho)$-weakly sharp and $\scrU_2$ in (i) satisfies $\scrU_2\cap \scrS_2\ne \emptyset$, then for the same set $\scrU_{1}$ as in (ii), we have the lower bound 
\begin{align}
-B_{\scrU_{1}}\left[\frac{\sup_{z\in\scrU_{2}}\norm{z^{1}-z}^{2}}{2T_{K}(\alpha/\rho)}+ \frac{C_{\scrU_{2}}\sum_{k=1}^{K}t_{k}\sigma_{k}}{(\alpha/\rho)T_{K}}\right]^{\frac{1}{\rho}}  \leq  \Theta_{\rm Opt}(\bar{z}^{K}\vert \scrU_{1}\cap\scrS_{2}). \label{eq:complexitySharp}
\end{align}
\end{itemize}
\end{theorem}

\begin{remark}
    Theorem \ref{th:Gap1} derives a rate in terms of the feasibility and the optimality gap of the hierarchical HVI problem \eqref{eq:P}. These rates are  established relative to arbitrary compact sets $\scrU_{1},\scrU_{2}\subset\dom(g_{1})\cap\dom(g_{2})$. Such bounds are meaningful because of Lemma \ref{lem:gap} and the boundedness of the sequence $(z^{k})_{k\geq 1}$, a fact we establish with the derivation of \eqref{eq:Feas_unbounded}. The constants involved in \eqref{eq:Feas_unbounded} and \eqref{eq:Opt_unbounded} depend on the Lipschitz modulus of the operators $\opF_{i}$, and diameter-like constants, which can be exhibited thanks to the boundedness of the trajectory. 
    The stepsize condition is satisfied whenever  
    $t_{k}\leq \frac{1}{4}\frac{1}{L_{\opF_{2}}+\sigma_{k}L_{\opF_{1}}}$. This can be achieved by a time-varying stepsize or, since $\sigma_k$ is decreasing, by a fixed stepsize, e.g., $t_k =t\leq \frac{1}{4} \frac{1}{L_{\opF_{2}}+\sigma_{1}L_{\opF_{1}}}$.
    \close
\end{remark}
Choosing a specific regularization sequence $(\sigma_{k})_{k\geq 1}$ allows us to turn the estimates from Theorem \ref{th:Gap1} into concrete convergence rates. In detail, we consider the constant step size policy $t_{k}=t$ and regularization sequence
\begin{equation}\label{eq:sigma}
\sigma_{k}\eqdef\frac{a}{(k+b)^{\delta}}\qquad a,b>0,\delta\in(0,1).
\end{equation}
\begin{corollary}\label{cor:Polynomialsigma} 
Let the same assumptions as in Theorem \ref{th:Gap1} hold. Let the regularization sequence $(\sigma_{k})_{k\geq 1}$ be chosen according to \eqref{eq:sigma}, and set $t_{k}=t\leq \frac{1}{4(L_{\opF_2}+\sigma_1L_{\opF_2})}$ for all $k\geq 1$. Then, we have 
\begin{align}
    &\Theta_{\rm Feas}(\bar{z}^{K}\vert\scrU_2)\leq \frac{\sup_{z\in\scrU_2}\norm{z^{1}-z}^{2}}{2K t}+ \frac{a C_{\scrU_2}}{(1-\delta)(K+b)^{\delta}}, \label{eq:lower_level_rate_unbounded_detlta} \\
    &-B_{\scrU_{1}}\left[\frac{\sup_{z\in\scrU_{2}}\norm{z^{1}-z}^{2}}{2Kt(\alpha/\rho)}+ \frac{aC_{\scrU_{2}}}{(\alpha/\rho)(1-\delta)(K+b)^{\delta}} \right]^{1/\rho} \nonumber\\
  & \hspace{8em}\overset{(*)}{\leq} \Theta_{\rm Opt}(\bar{z}^{K}\vert\scrU_1\cap\scrS_{2})\leq \frac{C_{\scrU_{1}}}{a(K+b)^{1-\delta}}, \label{eq:upper_level_rate_unbounded_detlta}
\end{align}
 where the inequality $(*)$ holds under the $(\alpha,\rho)$-weak sharpness assumption.
\end{corollary} 
\begin{remark}
Under the regularization sequence \eqref{eq:sigma} and $(\alpha,\rho)$-weak sharpness of  the lower-level solution set $\scrS_{2}$, a sufficient condition for Assumption \ref{ass:ACcondition} to hold is $1>\delta>1-\frac{1}{\rho}$ (cf. Remark \ref{rem:Tikhonov}). This  coincides with the analysis for the optimization setting in \cite[Theorem 2.3]{boct2025accelerating}. \close 
\end{remark}

Theorem \ref{th:Gap1} is instrumental to draw conclusions about the asymptotic behavior of the sequences generated by Algorithm \ref{alg:Alg1}. However, this statement requires one more (mild) assumption on the data defining problem \eqref{eq:P}. 
\begin{assumption}[hierarchical coherence]
  \label{ass:coherence}
 The non-smooth functions in the definition of \eqref{eq:P} satisfies $\dom(g_{2})\subseteq\dom(g_{1}).$ 
\end{assumption}

\begin{theorem}
\label{th:sequences} 
Consider problem \eqref{eq:P}. Let Assumptions \ref{ass:Mappings}-\ref{ass:ACcondition} and the hierarchic coherence condition Assumption \ref{ass:coherence} hold. Let $(z^{k})_{k\geq1}$ and $(z^{k+1/2})_{k\geq1}$ be generated by Algorithm \ref{alg:Alg1} with $16t_{k}^{2}L_{k}^2\leq 1$ for all $k\geq 1$.
Then, the following holds true:
\begin{itemize}
\item[(i)] Asymptotic feasibility: All weak accumulation points of $(z^{k+1/2})_{k\geq1}$ are contained in $\scrS_{2}$. 
\item[(ii)] Ergodic average optimality: If additionally Assumption~\ref{ass:BVg} holds, then the sequence $(\bar{z}^{k})_{k\geq1}$ converges weakly to an element in $\scrS_{1}$. 
\end{itemize}
\end{theorem}



\section{Proofs of the main results}
\label{sec:AllProofs}
\subsection{Energetic bounds} 
\label{sec:energy}
%

In this section, we develop the energy bounds for the sequences $(z^{k})_{k\geq 1}$, $(z^{k+1/2})_{k\geq 1}$ generated by the optimistic extragradient method. 
To simplify the derivations, we introduce the shortcut notation
\begin{align}
&E_k(z)\eqdef\frac{1}{2}\norm{z^{k}-z}^2\qquad \forall k \geq 1, \label{eq:shortcut_notation_E}\\
&\Psi_{k}(z)\eqdef \inner{\opV_{k}(z^{k+1/2}),z^{k+1/2}-z}+G_{k}(z^{k+1/2})-G_{k}(z)\qquad \forall k \geq 1, \label{eq:shortcut_notation_Psi} \\ 
&D_{k}\eqdef t^{2}_{k-1}\norm{\opV_{k-1}(z^{k-3/2})-\opV_{k-1}(z^{k-1/2})}^{2}\qquad \forall k \geq 2. \label{eq:shortcut_notation_D}
\end{align}

\begin{lemma}\label{lem:estimates}
Let $(\sigma_{k})_{k\geq 1} \subset \R_{\geq 0}$, $(t_{k})_{k\geq 1} \subset \R_{\geq 0}$ and  $z\in\scrZ$ be arbitrary. 
Then, for all $k\geq 1$ we have 
\begin{equation}\label{eq:bound1}
\begin{split}
E_{k+1}(z)\leq E_{k}(z)-\frac{1}{2}\norm{z^{k+1/2}-z^{k}}^{2}-\frac{1}{4}\norm{z^{k+1}-z^{k+1/2}}^{2}+D_{k+1}-t_{k}\Psi_{k}(z)
\end{split}
\end{equation}
If additionally $16 t^{2}_{k}L_{k}^{2}\leq 1$ and defining $D_1\eqdef 0$, we have, for all $k\geq 1$, 
\begin{equation}
\label{eq:Energy}
\begin{split}
E_{k+1}(z)+D_{k+1}&\leq 
E_{k}(z)+D_k -t_k \Psi_k(z) -\frac{1}{4}\norm{z^{k+1/2}-z^{k}}^{2}\\
&-\frac{1}{4}\norm{z^{k+1}-z^{k+1/2}}^{2}.
\end{split}
\end{equation}
\end{lemma}
\begin{proof}
We start with proving eq. \eqref{eq:bound1}. Applying \eqref{eq:prox} to the iterate $z^{k+1}$, we obtain for all $z\in\scrZ$
\begin{align}
\inner{z^{k+1}-z^{k},z^{k+1}-z}\leq \inner{t_{k}\opV_{k}(z^{k+1/2}),z-z^{k+1}}+t_{k}\left(G_{k}(z)-G_{k}(z^{k+1})\right).\label{eq:z1}
\end{align}
Doing the same for the iterate $z^{k+1/2}$, but substituting $z=z^{k+1}$, we get for all $z \in \scrZ$
\begin{equation}\label{eq:z12}
\begin{split}
\inner{z^{k+1/2}-z^{k},z^{k+1/2}-z^{k+1}}&\leq \inner{t_{k}\opV_{k}(z^{k-1/2}),z^{k+1}-z^{k+1/2}}\\
&+t_{k}\left(G_{k}(z^{k+1})-G_{k}(z^{k+1/2})\right).
\end{split}\end{equation}
Now, we observe that 
\begin{align}\label{eq:zk2}
\norm{z^{k+1}-z}^{2}&=\norm{z^{k}-z}^{2}-\norm{z^{k+1}-z^{k}}^{2}+2\inner{z^{k+1}-z^{k},z^{k+1}-z},\\
\norm{z^{k+1}-z^{k}}^{2}&=\norm{z^{k+1}-z^{k+1/2}}^{2}+\norm{z^{k+1/2}-z^{k}}^{2} \nonumber\\
&+2\inner{z^{k+1}-z^{k+1/2},z^{k+1/2}-z^{k}}. \nonumber
\end{align}

Combining these two identities with \eqref{eq:z1} and \eqref{eq:z12}, we can continue as 
\begin{align}
\norm{z^{k+1}-z}^{2}&=\norm{z^{k}-z}^{2}-\norm{z^{k+1}-z^{k+1/2}}^{2}-\norm{z^{k+1/2}-z^{k}}^{2}\nonumber\\
&-2\inner{z^{k+1}-z^{k+1/2},z^{k+1/2}-z^{k}}+2\inner{z^{k+1}-z^{k},z^{k+1}-z}\nonumber\\
&\leq\norm{z^{k}-z}^{2}-\norm{z^{k+1}-z^{k+1/2}}^{2}-\norm{z^{k+1/2}-z^{k}}^{2}\nonumber\\
&+ 2t_{k}\inner{\opV_{k}(z^{k-1/2}),z^{k+1}-z^{k+1/2}}+2t_{k}\inner{\opV_{k}(z^{k+1/2}),z-z^{k+1}}\nonumber\\
&+ 2t_{k}\left(G_{k}(z)-G_{k}(z^{k+1/2})\right). \label{eq:extragradient_proof_1}
\end{align}
Adding and subtracting $z^{k+1/2}$ in the second inner product, leaves us with
\begin{align}
&\norm{z^{k+1}-z}^{2}\leq\norm{z^{k}-z}^{2}-\norm{z^{k+1}-z^{k+1/2}}^{2}-\norm{z^{k+1/2}-z^{k}}^{2}\nonumber\\
& \quad + 2t_{k}\inner{\opV_{k}(z^{k-1/2})-\opV_{k}(z^{k+1/2}),z^{k+1}-z^{k+1/2}}\nonumber\\
&\quad +2t_{k}\inner{\opV_{k}(z^{k+1/2}),z-z^{k+1/2}} + 2t_{k}\left(G_{k}(z)-G_{k}(z^{k+1/2})\right). \label{eq:energy_important_111}
\end{align}
Applying Young's inequality $\inner{a,b}\leq \frac{c}{2}\norm{a}^{2}+\frac{1}{2c}\norm{b}^{2}$ with $c=2t_{k}$ and \eqref{eq:shortcut_notation_D}, we obtain 
\begin{align*}
&2t_{k}\inner{\opV_{k}(z^{k-1/2})-\opV_{k}(z^{k+1/2}),z^{k+1}-z^{k+1/2}} \\
&\leq 
2t^{2}_{k}\norm{\opV_{k}(z^{k-1/2})-\opV_{k}(z^{k+1/2})}^{2}+\frac{1}{2}\norm{z^{k+1}-z^{k+1/2}}^{2}\\
&=2D_{k+1}+\frac{1}{2}\norm{z^{k+1}-z^{k+1/2}}^{2}.
\end{align*}
Plugging this inequality and definitions  \eqref{eq:shortcut_notation_E}, \eqref{eq:shortcut_notation_Psi} into the previous inequality gives us \eqref{eq:bound1}.

We now verify eq. \eqref{eq:Energy}. Observe that by \eqref{eq:shortcut_notation_D}
\begin{align*}
D_{k+1}&=2t^{2}_{k}\norm{\opV_{k}(z^{k-1/2})-\opV_{k}(z^{k+1/2})}^{2}-D_{k+1}\leq 2t_k^2L_k^2\norm{z^{k+1/2}-z^{k-1/2}}^{2}-D_{k+1}\\
&\leq 4t_{k}^{2}L^{2}_{k}\norm{z^{k+1/2}-z^{k}}^{2}+4t_{k}^{2}L^{2}_{k}\norm{z^{k}-z^{k-1/2}}^{2}-D_{k+1},
\end{align*}
where the first inequality follows from the $L_{k}$-Lipschitz continuity of $\opV_{k}$, and the second inequality follows from $(a+b)^2\leq 2a^2+2b^2$. We can therefore continue, 
\begin{align}
&D_{k+1}-\frac{1}{2}\norm{z^{k+1/2}-z^{k}}^{2} \nonumber\\
&\leq 4t_{k}^{2}L^{2}_{k}\norm{z^{k}-z^{k-1/2}}^{2}-D_{k+1}+\left(4t_{k}^{2}L^{2}_{k}-\frac{1}{2}\right)\norm{z^{k+1/2}-z^{k}}^{2}\label{eq:important_recursion_1}\\
&\leq 4t_{k}^{2}L^{2}_{k}\norm{z^{k}-z^{k-1/2}}^{2}-D_{k+1}-\frac{1}{4}\norm{z^{k+1/2}-z^{k}}^{2},
\label{eq:important_recursion}
\end{align}
where the last inequality is due to the step size choice satisfying $4t_{k}^{2}L^{2}_{k}\leq\frac{1}{4}$. If $k\geq 2$, using the definition of $z^k$ and $z^{k-1/2}$, the non-expansiveness of the proximal mapping, 
and that the step size assumption $16 t^{2}_{k}L_{k}^{2}\leq 1$ yields  $4 t^{2}_{k}L^{2}_{k}\leq 1$, we obtain
\begin{align}
4t_k^2L_k^2\norm{z^k-z^{k-1/2}}^2\leq
4t^{2}_{k}L^{2}_{k} t^{2}_{k-1} \norm{\opV_{k-1}(z^{k-3/2})-\opV_{k-1}(z^{k-1/2})}^{2}\leq D_k.
\label{eq:per_iter_proof_1}
\end{align}
The same inequality holds for $k=1$ since $D_1=0$ and $z^{1/2}=z^1$.
Combining \eqref{eq:per_iter_proof_1} with \eqref{eq:important_recursion}, plugging the result into \eqref{eq:bound1}, using the notation \eqref{eq:shortcut_notation_Psi}, and rearranging terms, we obtain \eqref{eq:Energy}.
\end{proof}

\subsection{Towards proving Theorem \ref{th:Gap1}} 
\label{sec:proofsmain} 
%
To obtain a bound on the feasibility gap, we start with an Opial-like Lemma.
\begin{lemma}\label{lem:boundsequence}
   Let Assumptions \ref{ass:Mappings}-\ref{ass:ACcondition} hold. Let $z^{*}\in\scrS_{1}$ be arbitrary and let $(z^{k})_{k\geq1}$ and $(z^{k+1/2})_{k\geq1}$ be generated by Algorithm \ref{alg:Alg1} with $16 t^{2}_{k}L_{k}^{2}\leq 1$. Define $D_1=0$ and
    \begin{equation}
    \label{eq:W_k_def}
        W_{k}(z)\eqdef E_{k}(z)+D_{k}\qquad \forall k\geq 1. 
    \end{equation}
    Then, the following statements hold:
    \begin{enumerate}
        \item[(a)] $\lim_{k\to\infty}W_{k}(z^{*})$ exists in $\R$; \label{boundsequence_1}
        \item[(b)] $\displaystyle\lim_{k\to\infty}\norm{z^{k+1/2}-z^{k}}=\lim_{k\to\infty}\norm{z^{k+1}-z^{k+1/2}}=\lim_{k\to\infty}\norm{z^{k+1/2}-z^{k-1/2}}=0$; \label{boundsequence_2}
        \item[(c)] $(z^{k})_{k\geq 1}$ and $(z^{k+1/2})_{k\geq 1}$ \label{boundsequence_3} are bounded.
        \item[(d)] $\lim_{k\to\infty}\norm{z^{k}-z^{\ast}}$ and $\lim_{k\to\infty}\norm{z^{k+1/2}-z^{*}}$ both exist and are equal.
    \end{enumerate}
\end{lemma}
\begin{proof}
    Pick $z^{\ast}\in\scrS_{1}\subseteq\scrS_{2}$ so that there exists $p^{\ast}\in\NC_{\scrS_{2}}(z^{\ast})$ s.t. $-\opF_{1}(z^{\ast})-p^{\ast}\in\partial g_{1}(z^{\ast})$. Using the monotonicity of $\opF_{1}$ and the convex subgradient inequality for $g_{1}$, we obtain 
\begin{align*}
&-H^{(\opF_{1},g_{1})}(z^{\ast},z^{k+1/2})+\inner{p^{\ast},z^{k+1/2}-z^{\ast}}\\
&=\inner{\opF_{1}(z^{k+1/2}),z^{k+1/2}-z^{\ast}}+g_{1}(z^{k+1/2})-g_{1}(z^{\ast})+\inner{p^{\ast},z^{k+1/2}-z^{\ast}}\\
&\geq \inner{\opF_{1}(z^{\ast}),z^{k+1/2}-z^{\ast}}+g_{1}(z^{k+1/2})-g_{1}(z^{\ast})+\inner{p^{\ast},z^{k+1/2}-z^{\ast}}\geq 0.
\end{align*}
Combining with the energy inequality \eqref{eq:Energy} and the fact that 
\begin{align}
\Psi_k(z)=-\sigma_kH^{(\opF_{1},g_{1})}(z,z^{k+1/2})-H^{(\opF_{2},g_{2})}(z,z^{k+1/2}),\label{eq:Psi_H}
\end{align}
we continue
\begin{align}
&E_{k+1}(z^{\ast})+D_{k+1} \nonumber\\
&\leq E_{k+1}(z^{\ast})+D_{k+1}+t_{k}\sigma_{k}\left(-H^{(\opF_{1},g_{1})}(z^{\ast},z^{k+1/2})+\inner{p^{\ast},z^{k+1/2}-z^{\ast}}\right)\nonumber\\
&\leq E_{k}(z^{\ast})+D_{k} -\frac{1}{4}\norm{z^{k+1/2}-z^{k}}^{2}-\frac{1}{4}\norm{z^{k+1}-z^{k+1/2}}^{2}\nonumber\\
&+t_{k}\left(H^{(\opF_{2},g_{2})}(z^{\ast},z^{k+1/2})+\inner{\sigma_{k}p^{\ast},z^{k+1/2}-z^{\ast}}\right).\label{eq:energy2}
\end{align}
Since $p^{\ast}\in \NC_{\scrS_{2}}(z^{\ast})$, we know from \eqref{eq:support_equality} that $\inner{\sigma_kp^*,z^*}=\supp(\sigma_kp^*|\scrS_2)$ and 
\begin{align*}
&H^{(\opF_{2},g_{2})}(z^{\ast},z^{k+1/2})+\inner{\sigma_{k}p^{\ast},z^{k+1/2}-z^{\ast}}\\
&=\inner{\sigma_{k}p^{\ast},z^{k+1/2}}+H^{(\opF_{2},g_{2})}(z^{\ast},z^{k+1/2})-\supp(\sigma_{k}p^{\ast}\vert \scrS_{2})\\
&\leq \sup_{z\in\dom(g_{2})}\left(\inner{\sigma_{k}p^{\ast},z}+H^{(\opF_{2},g_{2})}(z^{\ast},z)\right)-\supp(\sigma_{k}p^{\ast}\vert\scrS_{2})\\
&\overset{\eqref{eq:phimaintext}}{=}\varphi^{(\opF_{2},g_{2})}(z^{\ast},\sigma_{k}p^{\ast})-\supp(\sigma_{k}p^{\ast}\vert \scrS_{2}) \leq\sup_{z\in\scrS_{2}}\varphi^{(\opF_{2},g_{2})}(z,\sigma_{k}p^{\ast})-\supp(\sigma_{k}p^{\ast}\vert \scrS_{2}) \eqdef c_{k}.
\end{align*}
We note that \eqref{eq:FPlower} in Appendix \ref{app:FP} guarantees that the upper bound is informative, in the sense that $c_{k}\geq 0$. Plugging this back into \eqref{eq:energy2} and using \eqref{eq:W_k_def}, we obtain
\[
W_{k+1}(z^{\ast}) \leq W_{k}(z^{\ast}) -\frac{1}{4}\norm{z^{k+1/2}-z^{k}}^{2}-\frac{1}{4}\norm{z^{k+1}-z^{k+1/2}}^{2}+t_{k}c_{k}.
\]
The summability condition \eqref{eq:AttCza} implies $\sum_{k}t_{k}c_{k}<\infty$. Hence, the non-negativity of $W_k(z^*)$ allows us to deduce from Lemma \ref{lem:RS} that $\lim_{k\to\infty}W_{k}(z^{\ast})=W_{\infty}(z^{\ast})$ exists, and $\lim_{k\rightarrow\infty} \norm{z^{k+1/2}-z^k}= \lim_{k\to\infty}\norm{z^{k+1}-z^{k+1/2}}=0$. By the triangle inequality, the latter also implies that $\lim_{k\to\infty}\norm{z^{k+1/2}-z^{k-1/2}}=0$. By the definition of $E_k(z^*)$, we further obtain that the sequences $(z^{k})_{k\geq1}$ and $(z^{k+1/2})_{k\geq1}$ are bounded.

By \eqref{eq:shortcut_notation_D}, $L_k$-Lipschitzness of $\opV_k$ and item (b), we have, as $k \to  \infty$, 
\begin{align*}
    0\leq D_{k+1}&=t_{k}^{2}\norm{\opV_{k}(z^{k-1/2})-\opV_{k}(z^{k+1/2})}^{2}\leq t_{k}^{2}L^{2}_{k}\norm{z^{k+1/2}-z^{k-1/2}}^{2} \to 0. 
\end{align*}
Since $E_k(z^*)=W_k(z^*)-D_k$, it thus follows $\lim_{k\rightarrow\infty}E_k(z^*)=W_{\infty}(z^{*}).$ Moreover, from $\norm{z^k-z^{k+1/2}}\to 0$, and  \eqref{eq:shortcut_notation_E}, we immediately deduce
\begin{align*}\lim_{k\rightarrow \infty}\norm{z^{k+1/2}-z^*}^2&=\lim_{k\rightarrow\infty}\norm{z^k-z^*}^2+\norm{z^k-z^{k+1/2}}^2+\inner{z^k-z^{k+1/2},z^k-z^*}\\
&=\lim_{k\rightarrow\infty}\norm{z^k-z^*}^2=2W_{\infty}(z^{*}).
\end{align*}
This proves (d).
\end{proof}

If Assumption \ref{ass:ACcondition}  holds, thanks to Lemma \ref{lem:boundsequence} there exists $R>0$ such that $z^{k},z^{k+1/2}\in\ball(z^{1},R)$ for all $k\geq 1$. 
If $\dom(g_{1})\cap\dom(g_{2})$ is compact, such a ball also exists by construction of the algorithm, since then the iterates are confined to stay in the compact set $\dom(g_{1})\cap\dom(g_{2}).$ We denote $\scrD_i\eqdef\dom(g_{i})$, $i=1,2$.

Consider a point $z\in \scrD_1\cap\scrD_2$. Invoking Assumption~\ref{ass:Mappings}, we have
\begin{align*}
    &\inner{\opF_{1}(z^{k+1/2}),z-z^{k+1/2}}  \leq    \norm{\opF_{1}(z^{k+1/2})}\cdot\norm{z-z^{k+1/2}} \nonumber\\
    &\leq \left(\norm{\opF_{1}(z^{1})}+\norm{\opF_{1}(z^{k+1/2})-\opF_{1}(z^{1})}\right)\cdot\left(\norm{z^1-z^{k+1/2}}+\norm{z^1}+\norm{z}\right)\nonumber\\
    &\leq   \left(\norm{\opF_{1}(z^{1})}+L_{\opF_{1}} R\right)\cdot(R+\norm{z^{1}}+\norm{z})\triangleq a(z),
\end{align*}
where the last inequality uses that $\opF_1$ is $L_{\opF_{1}}$-Lipschitz and $z^{k+1/2} \in\ball(z^{1},R)$ for all $k\geq 1$.
Moreover, defining the compact set $\scrB_{R}\eqdef\ball(z^{1},R)\cap\scrD_1\cap\scrD_2$, by Assumption~\ref{ass:BVg},
\begin{align*}
    g_{1}(z)-g_{1}(z^{k+1/2})&
    \leq \Var(g_{1}\vert\{z\}\times \scrB_{R})\triangleq b(z).
\end{align*}
Thus, defining $C_z\triangleq a(z)+b(z)$ we can conclude that
\begin{align}
H^{(\opF_{1},g_{1})}(z,z^{k+1/2})=\inner{\opF_{1}(z^{k+1/2}),z-z^{k+1/2}} + g_{1}(z)-g_{1}(z^{k+1/2}) \leq C_z. \label{eq:H_1_bound}
\end{align}
By \eqref{eq:Psi_H} and monotonicity of $\opF_2$ we have, for all $z \in \scrD_1\cap\scrD_2$  and $k\geq 1$,
\begin{align}
\Psi_{k}(z) + \sigma_{k}C_z
&\geq \Psi_{k}(z) + \sigma_{k} H^{(\opF_{1},g_{1})}(z,z^{k+1/2}) \nonumber\\
&= -H^{(\opF_{2},g_{2})}(z,z^{k+1/2})\geq H^{(\opF_{2},g_{2})}(z^{k+1/2},z). \label{eq:lower_level_proof_4}
\end{align}
Multiplying by $t_k$, using  \eqref{eq:W_k_def} and \eqref{eq:Energy},  where we now neglect the non-positive quadratic terms, we obtain, for all $z \in \scrD_1\cap\scrD_2$ and $k\geq 1$,
\begin{align}
t_kH^{(\opF_{2},g_{2})}(z^{k+1/2},z) &\leq 
t_k\Psi_{k}(z) + t_k\sigma_{k}C_{\scrU_2}  \leq W_{k}(z)-W_{k+1}(z)+ t_k\sigma_{k}C_z.
\label{eq:lower_level_proof_5}
\end{align}
Summing \eqref{eq:lower_level_proof_5} for $k=1,\ldots,K$ 
dividing by $T_K$, 
and applying $D_1=0$, yields 
\begin{equation}
\label{eq:sum_lower_level_1}
\sum_{k=1}^{K}\frac{t_k}{T_{K}}H^{(\opF_{2},g_{2})}(z^{k+1/2},z)\leq \frac{E_{1}(z)}{T_{K}}  + \frac{\sum_{k=1}^{K}t_{k} \sigma_{k}}{T_{K}} C_z \quad \forall z \in \scrD_1\cap\scrD_2. 
\end{equation}
Using the definition of the ergodic trajectory \eqref{eq:ergodic}, we can apply the Jensen inequality in order to finally arrive at 
\begin{align}\label{eq:lower_level_ergodic_bound_z}
    H^{(\opF_{2},g_{2})}(\bar{z}^{K},z)\leq \frac{E_{1}(z)}{T_{K}}  + \frac{\sum_{k=1}^{K}t_k \sigma_k}{T_{K}} C_z \qquad \forall z \in \dom(g_{1})\cap\dom(g_{2}).
\end{align}
Since $\scrU_2 \subset \dom(g_{1})\cap\dom(g_{2})$ is compact, we have that $b(z)\leq \Var(g_{1}\vert\scrU_{2}\times \scrB_{R})$ and $C_{\scrU_2}\eqdef \sup_{z\in\scrU_2}C_z$, as well as $\sup_{z\in\scrU_2} E_{1}(z)$, are well-defined. Thus, for any compact subset $\scrU_2\subset\dom(g_{1})\cap\dom(g_{2}) \subset\scrZ$ we can use the definition of the localized gap function \eqref{eq:gap_Def_new} specialized as the feasibility gap \eqref{eq:FeasGap}, to obtain \eqref{eq:Feas_unbounded}.





\paragraph{Proof of eq. \eqref{eq:lower_level_rate_unbounded_detlta}}
Choosing constant step $t_{k}=t$ and penalty sequence as in \eqref{eq:sigma}, we easily see that $\frac{\sum_{k=1}^{K}t_{k} \sigma_{k}}{T_{K}}\leq \frac{a}{(1-\delta)(K+b)^{\delta}},$ so that \eqref{eq:Feas_unbounded} transforms into \eqref{eq:lower_level_rate_unbounded_detlta}.
\begin{remark}
Corollary \ref{cor:Polynomialsigma} is formulated for regularization sequences $(\sigma_{k})_{k\geq 1}$ of the form \eqref{eq:sigma} with $\delta\in(0,1)$. In the limiting case $\delta=1$ and constant step size $t_{k}=t$, the last term in \eqref{eq:Feas_unbounded} is upper bounded as $\frac{a\log(K+b)}{K}C_{\scrU_2}$.
Hence, even in this limiting case, not satisfying condition \eqref{eq:tsigma} required in Algorithm \ref{alg:Alg1}, the ergodic average $(\bar{z}^{k})_{k\geq1}$ asymptotically approaches the feasible set $\scrS_{2}$.\close
\end{remark}

\paragraph{Establishing the rate on the optimality gap \eqref{eq:Opt_unbounded}}
\label{sec:optimality}
Let $\scrU_{1}\subset\scrD_1\cap\scrD_2$ be a compact set s.t. $\scrU_{1}\cap\scrS_{1}\neq\emptyset$. 
By the monotonicity of $\opF_{1}$ and $\opF_{2}$ and choosing $z\in\scrS_2$, we see 
\begin{align*}
\Psi_{k}(z)&\overset{\eqref{eq:Psi_H}}{=} -\sigma_kH^{(\opF_{1},g_{1})}(z,z^{k+1/2})-H^{(\opF_{2},g_{2})}(z,z^{k+1/2}) \\ &\geq \sigma_kH^{(\opF_{1},g_{1})}(z^{k+1/2},z)+H^{(\opF_{2},g_{2})}(z^{k+1/2},z) \geq \sigma_kH^{(\opF_{1},g_{1})}(z^{k+1/2},z).
\end{align*}
Plugging this in \eqref{eq:Energy}, dropping the non-positive quadratic terms, and recalling \eqref{eq:W_k_def}, we arrive at 
$\sigma_{k}t_{k}H^{(\opF_{1},g_{1})}(z^{k+1/2},z)\leq W_{k}(z)-W_{k+1}(z)$ for all $z \in \scrS_2$.

Since by the $L_{k-1}$-Lipschitz property of $\opV_{k-1}$,
\begin{align*}
        W_{k}(z)
        &\leq \frac{1}{2}\norm{z^{k}-z}^{2}+t^{2}_{k-1}L^{2}_{k-1}\norm{z^{k-3/2}-z^{k-1/2}}^{2},
\end{align*}
we immediately see that, for any fixed $z\in \scrS_2$, $W_{k}(z)$ is bounded in the sense that for some positive constant  $\bar{\omega}(z)$, $W_{k}(z)\leq\bar{\omega}(z)$ for all $k\geq 1$.
Using this bound, $W_k(z)\geq 0$, and $\sigma_{k+1}\leq\sigma_{k}$, we obtain that, for any $z \in \scrS_2$, 
\begin{align*}
    \sum_{k=1}^{K}t_{k}H^{(\opF_{1},g_{1})}(z^{k+1/2},z)&\leq \frac{1}{\sigma_{1}}W_{1}(z)+W_{2}(z)(\frac{1}{\sigma_{2}}-\frac{1}{\sigma_{1}})+\ldots+W_{K}(z)(\frac{1}{\sigma_{K}}-\frac{1}{\sigma_{K-1}})\\
   & \leq \frac{\bar{\omega}(z)}{\sigma_{K}}.
\end{align*}
Dividing both sides by $T_{K}$ and applying the Jensen inequality, we arrive at 
\begin{equation}
\label{eq:upper_level_rate_unbounded_proof1}
      H^{(\opF_{1},g_{1})}(\bar{z}^{K},z)  \leq \frac{\bar{\omega}(z)}{T_{K}\sigma_{K}} \quad \forall z \in \scrS_2.
\end{equation}
Thus, for any compact subset $\scrU_1\subset\scrZ$ with $\scrU_{1}\cap\scrS_{2}\neq\emptyset$, we can use the definition of the localized gap function \eqref{eq:gap_Def_new} specialized as the optimality gap \eqref{eq:OptGap} to obtain $\Theta_{\rm Opt}(\bar{z}^{K}\vert\scrU_1 \cap \scrS_{2})\leq \frac{C_{\scrU_1}}{T_K\sigma_{K}}$,
where $C_{\scrU_{1}}\eqdef \sup_{z\in\scrU_{1}\cap\scrS_{2}}\bar{\omega}(z)$. Note that the latter is well-defined since $\scrS_{1}\subseteq\scrS_2$  and $\scrU_{1}\cap\scrS_{1}\neq\emptyset$ by the Theorem assumption. Thus, we have obtained the upper bound in \eqref{eq:Opt_unbounded}.
The lower bound in \eqref{eq:Opt_unbounded} holds since $\scrS_1 \subseteq \scrS_2$:
\[
-B_{\scrU_1}\dist(\bar{z}^{K},\scrS_{2})\stackrel{\eqref{eq:LB1}}{\leq}\Theta_{\rm Opt}(\bar{z}^{K}\vert\scrU_{1}\cap\scrS_{1}) \leq \Theta_{\rm Opt}(\bar{z}^{K}\vert\scrU_{1}\cap\scrS_{2}).
\]



\paragraph{Establishing the improved rate on the optimality gap \eqref{eq:complexitySharp}}
If additionally the lower level solution set enjoys $(\alpha,\rho)$-weak sharpness we have from \eqref{eq:WS} and \eqref{eq:Feas_unbounded}
\begin{align*}
\Theta_{\rm Opt}(\bar{z}^{K}\vert\scrU_{1}\cap\scrS_{2}&)\geq
-B_{\scrU_{1}}\dist(\bar{z}^{K},\scrS_{2})\\
&\geq -B_{\scrU_{1}}\left[\frac{\rho}{\alpha}\Theta_{\rm Feas}(\bar{z}^{K}\vert\scrU_{2}\cap\scrS_{2})\right]^{1/\rho} \geq -B_{\scrU_{1}}\left[\frac{\rho}{\alpha}\Theta_{\rm Feas}(\bar{z}^{K}\vert\scrU_{2})\right]^{1/\rho}\\
&\geq-B_{\scrU_{1}}\left[\frac{\sup_{z\in\scrU_{2}}\norm{z^{1}-z}^{2}}{2T_{K}(\alpha/\rho)}+ \frac{C_{\scrU_{2}}\sum_{k=1}^{K}t_{k}\sigma_{k}}{(\alpha/\rho)T_{K}}\right]^{1/\rho}.
\end{align*}
\paragraph{Proof of \eqref{eq:upper_level_rate_unbounded_detlta}}
Choosing the penalty sequence $(\sigma_{k})_{k\geq 1}$ according to \eqref{eq:sigma} and the step size $t_{k}=t$,  
the r.h.s. of \eqref{eq:Opt_unbounded} is upper bounded by the r.h.s. in \eqref{eq:upper_level_rate_unbounded_detlta} and the l.h.s. in \eqref{eq:upper_level_rate_unbounded_detlta} follows from \eqref{eq:lower_level_rate_unbounded_detlta}.
Note that if $\delta=1$ in the choice of the penalty sequence \eqref{eq:sigma}, we do not obtain convergence in terms of the gap function to $0$, but rather only an $O(1)$ upper bound and a $o(1)$ lower bound.

\paragraph{Proof of Theorem \ref{th:sequences}} 
(i) From the prox-inequality \eqref{eq:prox} applied to the point $z^{k+1/2}$ and divided by $t_k>0$, we have, for any $z\in \scrZ$, 
\begin{align}
t_k^{-1}\inner{z^{k+1/2}-z^{k},z^{k+1/2}-z}&\leq \inner{\opV_{k}(z^{k-1/2}),z-z^{k+1/2}}
+ G_{k}(z)-G_{k}(z^{k+1/2})\nonumber\\
&\hspace{-2em}\overset{\eqref{eq:VkGk_def},\eqref{eq:HFG_def}}{=}\sigma_{k}H^{(\opF_{1},g_{1})}(z,z^{k+1/2})+H^{(\opF_{2},g_{2})}(z,z^{k+1/2})+e_{k} \nonumber \\
&\hspace{-2em} \leq   \sigma_{k}H^{(\opF_{1},g_{1})}(z,z^{k+1/2})-H^{(\opF_{2},g_{2})}(z^{k+1/2},z)+ e_{k},\label{eq:asympt_conv_proof_1}
\end{align}
where $e_{k}\eqdef \inner{\opV_{k}(z^{k-1/2})-\opV_{k}(z^{k+1/2}),z-z^{k+1/2}}$. Since $\opV_{k}$ is $L_k$-Lipschitz, by Lemma \ref{lem:boundsequence} (b) and (c), we have 
$\abs{e_{k}}\leq L_k\norm{z^{k-1/2}-z^{k+1/2}}\cdot \norm{z-z^{k+1/2}} \to 0$ as $k \to \infty$. Again, by Lemma \ref{lem:boundsequence} (b) and (c), and since $t_k>0$, we have $t_k^{-1}\inner{z^{k+1/2}-z^{k},z^{k+1/2}-z}\to 0$ as $k \to \infty$. 
By Lemma \ref{lem:boundsequence} (c), the sequence  $(z^{k+1/2})_{k\geq1}$ is bounded. Hence, we can find a weak accumulation point $\tilde{z}$. Let $\tilde{z}$ be such that $z^{k_j+1/2}\wlim \tilde{z}$ as $j\to \infty$. By monotonicity of $\opF_{1}$, lower semi-continuity of $g_{1}$, we get, for any $z \in \dom(g_1)$,  
\begin{align*}
    &\lim_{j\to\infty}H^{(\opF_{1},g_{1})}(z,z^{k_j+1/2}) = \lim_{j\to\infty} \inner{\opF_{1}(z^{k_j+1/2}),z-z^{k_j+1/2}}+g_{1}(z)-g_{1}(z^{k_j+1/2})\\
    &\leq \lim_{j\to\infty}\inner{\opF_{1}(z),z-z^{k_j+1/2}}+g_{1}(z)-g_{1}(z^{k_j+1/2}) \leq \inner{\opF_{1}(z),z-\tilde{z}}+g_{1}(z)-g_{1}(\tilde{z}). 
\end{align*}
Thus, since $\sigma_k\to0$ as $k\to \infty$, we have for any $z \in \dom(g_2)\subseteq \dom(g_1)$, that $\lim_{j\to\infty}\sigma_{k_j}H^{(\opF_{1},g_{1})}(z,z^{k_j+1/2})\leq 0.$ 
Thus, we obtained the limits of all the terms in \eqref{eq:asympt_conv_proof_1} except the term with $H^{(\opF_{2},g_{2})}$. Using these limits when passing in \eqref{eq:asympt_conv_proof_1} to the subsequence indexed by $k_j$ and then to the limit $j\to \infty$, using the lower semicontinuity of $g_2$, we obtain, for any $z \in \dom(g_2)$ 
\begin{align*}
    0 &\leq \lim_{j\to\infty}-H^{(\opF_{2},g_{2})}(z^{k_j+1/2},z)
    = \lim_{j\to\infty} \inner{\opF_{2}(z),z-z^{k_j+1/2}}+g_{2}(z)-g_{2}(z^{k_j+1/2})\\
    &\leq \inner{\opF_{2}(z),z-\tilde{z}}+g_{2}(z)-g_{2}(\tilde{z}). 
\end{align*}
For $z \notin \dom(g_2)$ the above inequality holds trivially. Thus, it holds for all $z \in \scrZ$. This shows that any weak limit point of $(z^{k+1/2})_{k\geq1}$ is a Minty solution of $\HVI(\opF_{2},g_{2})$. Under Assumption \ref{ass:Mappings}, Minty solutions and Stampacchia solutions of $\HVI(\opF_{2},g_{2})$ \eqref{eq:VI_general} coincide (see also the proof of Lemma \ref{lem:gap} in Appendix \ref{app:proofs}), showing that $\tilde{z}\in\scrS_{2}$.

(ii) By Lemma \ref{lem:boundsequence} (c), the sequence  $(z^{k+1/2})_{k\geq1}$ is bounded. Hence the ergodic average sequence $(\bar{z}^{k})_{k\geq1}$ is bounded as well and thus admits weakly converging subsequences. Let $\tilde{z}$ be such that $\bar{z}^{k_j}\wlim \tilde{z}$ as $j\to \infty$.
For any reference point $z\in \dom (g_1) \cap \dom (g_2)=\dom (g_2)$ applying condition \eqref{eq:tsigma} to \eqref{eq:lower_level_ergodic_bound_z}, we obtain that $\lim_{j\rightarrow \infty} H^{(\opF_{2},g_{2})}(\bar{z}^{k_j},z)\leq 0$. Thus, using the lower semi-continuity of $g_2$, $\tilde{z}$ satisfies $H^{(\opF_{2},g_{2})}(\tilde{z},z)\leq 0$. This inequality trivially holds for $z\notin \dom (g_2)$, and hence holds for all $z \in \scrZ$. Thus, $\tilde{z}$ is a Minty solution of $\HVI(\opF_{2},g_{2})$, which again, under Assumption  \ref{ass:Mappings} is also a Stampacchia solution of $\HVI(\opF_{2},g_{2})$. This implies $\tilde{z}\in \scrS_2$. 

For any point $z\in\scrS_{2}$, \eqref{eq:upper_level_rate_unbounded_proof1} and \eqref{eq:tsigma} yield that
$\lim_{j\rightarrow \infty} H^{(\opF_{1},g_{1})}(\bar{z}^{k_j},z) \leq 0.$ 
This, by the lower semi-continuity of $g_{1}$, implies 
\begin{align*}
    \inner{\opF_{1}(z),\tilde{z}-z}+&g_{1}(\tilde{z})-g_{1}(z) \leq \lim_{j\to\infty}\left(\inner{\opF_{1}(z),\bar{z}^{k_j}-z}+g_{1}(\bar{z}^{k_j})-g_{1}(z)\right)\leq 0. 
\end{align*} 
Since $z\in\scrS_{2}$ is arbitrary and $\tilde{z}\in \scrS_2$, we have that $\tilde{z}$ is a Minty solution of \eqref{eq:P}, which is, under Assumption  \ref{ass:Mappings}, also a Stampacchia solution of $\HVI(\opF_{1},g_{1})$, i.e., $\tilde{z}\in\scrS_{1}$. 

Lemma \ref{lem:boundsequence} (d) shows that the limit of $(\norm{z^{k+1/2}-z})_{k\geq1}$ exists for all $z\in\scrS_{1}$. As just demonstrated, all weak accumulation points of $(\bar{z}^{k})_{k\geq1}$ are points in $\scrS_{1}$. Lemma \ref{lem:Opial} thus shows that $\bar{z}^{k}\wlim z_{\infty}$ for some $z_{\infty}\in\scrS_{1}$.

\section{Extensions} 
\label{sec:extensions}
%

In this section, first we develop estimates for the case in which the upper-level $\HVI$ is strongly monotone, implying that the entire problem \eqref{eq:P} has a unique solution. Second, we show how our method can be extended in a straightforward manner to the forward-backward-forward splitting method of \cite{Tsen00}. 

\subsection{Improved rates under strong monotonicity assumption} 
\label{sec:linear}
We obtain faster rates of convergence under strong monotonicity of the hierarchical VI. 
\begin{assumption}\label{ass:SM}
The operator $\opF_{1}:\scrZ\to\scrZ$ is $\mu$-strongly monotone.
\end{assumption}
Since we assume that the whole hierarchical equilibrium problem \eqref{eq:P} has a nonempty feasible set, Assumption \ref{ass:SM} implies that $\scrS_{1}=\{z^{*}\}$ for some $z^{*}\in\scrS_{2}$. We again establish rates on gap functions in terms of the suitably adapted ergodic average 
\begin{equation}\label{eq:Ergodic_SM}
\bar{z}^K=\frac{\sum_{i=1}^Kt_{i}\sigma_{i}\gamma_iz^{i+1/2}}{\sum_{i=1}^Kt_{i}\sigma_{i}\gamma_i},\text{ where }\gamma_k\eqdef \frac{1}{\prod_{i=1}^k(1-t_{i}\sigma_{i}\mu)},\gamma_0\eqdef 1. 
\end{equation}

\begin{theorem}\label{thm:ComplexitySM}
Consider problem \eqref{eq:P}. Let Assumptions \ref{ass:Mappings}, \ref{ass:CQ}, \ref{ass:BVg} and Assumption \ref{ass:SM} hold.
Let additionally either Assumption \ref{ass:ACcondition} hold, or $\dom(g_{1})\cap\dom(g_{2})$ be compact.
Let $(\bar{z}^{k})_{k\geq1}$ as in \eqref{eq:Ergodic_SM} be generated by Algorithm \ref{alg:Alg1} with
\begin{equation}\label{eq:StepSM}
8t^{2}_{k}L^{2}_{k}+2t_{k}\sigma_{k}\mu\leq 1, \quad k\geq 1.
\end{equation}
Let $\scrU_{2}\subset\dom(g_{1})\cap\dom(g_{2})$ be a nonempty compact set. Then, there exists a constant $\tilde{C}_{\scrU_{2}}>0$ for which 
\begin{equation}
    \label{eq:Feas_SM}
    \Theta_{\rm Feas}(\bar{z}^{K}\vert\scrU_2)\leq \frac{\sigma_1\sup_{z\in\scrU_2}\norm{z^{1}-z}^{2}}{2\sum_{i=1}^{K}t_{i}\sigma_{i}\gamma_{i}}+\tilde{C}_{\scrU_{2}} \frac{\sum_{i=1}^{k}t_i \sigma^{2}_i \gamma_{i}}{\sum_{i=1}^{K}t_{i}\sigma_{i}\gamma_{i}}.
\end{equation}
Additionally, for any nonempty compact subset $\scrU_{1}\subset\dom(g_{1})\cap\dom(g_{2})$ with $\scrU_{1}\cap\scrS_{1}\neq\emptyset$, there exists a constant $B_{\scrU_{1}}>0$ such that  
\begin{equation}\label{eq:Opt_unbounded_SM}
-B_{\scrU_{1}}\dist(\bar{z}^{K},\scrS_{2})\leq \Theta_{\rm Opt}(\bar{z}^{K}\vert\scrU_{1}\cap\scrS_{2})\leq \frac{ \sup_{z\in\scrU_{1}\cap\scrS_{2}}\norm{z^{1}-z}^{2}}{2\sum_{i=1}^{K}t_{i}\sigma_{i}\gamma_{i}}. 
\end{equation}
If additionally the lower level solution set $\scrS_{2}$ is $(\alpha,\rho)$-weakly sharp and $\scrU_2$ is chosen such that $\scrU_2\cap \scrS_2\ne \emptyset$, then for the same set $\scrU_{1}$
\begin{equation}\label{eq:complexitySharp_SM}\begin{split}
-B_{\scrU_{1}}&\left[\frac{\sigma_1\sup_{z\in\scrU_{2}}\norm{z^{1}-z}^{2}}{2(\alpha/\rho)\sum_{i=1}^{K}t_{i}\sigma_{i}\gamma_{i}}+ \frac{\tilde{C}_{\scrU_{2}}\sum_{i=1}^{k}t_i \sigma^{2}_i \gamma_{i}}{(\alpha/\rho)\sum_{i=1}^{K}t_{i}\sigma_{i}\gamma_{i}}\right]^{\frac{1}{\rho}}\leq  \Theta_{\rm Opt}(\bar{z}^{K}\vert \scrU_{1}\cap\scrS_{2}).
\end{split}\end{equation}
\end{theorem}
\begin{proof}
We start with a refined version of the recursion \eqref{eq:Energy}.
Recall the  bound \eqref{eq:bound1} that implies for any $z \in \scrZ$ and $k\geq 1$:
\begin{align}
E_{k+1}(z)&\leq E_{k}(z)-\frac{1}{2}\norm{z^{k+1/2}-z^{k}}^{2}+D_{k+1} \notag\\
&+t_{k}\inner{\opV_{k}(z^{k+1/2}),z-z^{k+1/2}}+t_{k}\left(G_{k}(z)-G_{k}(z^{k+1/2})\right). \label{eq:SM_proof_1}
\end{align}
For the second and third terms in the r.h.s. we have from \eqref{eq:important_recursion_1}
\begin{align}
D_{k+1}-\frac{1}{2}\norm{z^{k+1/2}-z^{k}}^{2} &\leq 
4t_{k}^{2}L^{2}_{k}\norm{z^{k}-z^{k-1/2}}^{2}-D_{k+1}\nonumber\\
&+\left(4t_{k}^{2}L^{2}_{k}-\frac{1}{2}\right)\norm{z^{k+1/2}-z^{k}}^{2}.\label{eq:SM_proof_2}
\end{align}
Acting in the same way as for the proof of \eqref{eq:per_iter_proof_1}, but with the new step size condition \eqref{eq:StepSM} that implies $4t_k^2L_k^2\leq 1-t_{k}\sigma_{k}\mu$, we further obtain
\begin{equation}\label{eq:SM_proof_3}
\begin{split}4t^{2}_{k}L^{2}_{k}\norm{z^{k}-z^{k-1/2}}^{2}&\leq 4t^{2}_{k}L_{k}^{2}t_{k-1}^{2}\norm{\opV_{k-1}(z^{k-1/2})-\opV_{k-1}(z^{k-3/2})}^{2}\\
&\leq (1-t_{k}\sigma_{k}\mu)D_k.
\end{split}
\end{equation}
Combining \eqref{eq:SM_proof_2}, \eqref{eq:SM_proof_3}, we obtain
\begin{equation}\label{eq:SM_proof_4}
\begin{split}
D_{k+1}-\frac{1}{2}\norm{z^{k+1/2}-z^{k}}^{2} &\leq 
(1-t_{k}\sigma_{k}\mu)D_k-D_{k+1}  \\
&+\left(4t_{k}^{2}L^{2}_{k}-\frac{1}{2}\right)\norm{z^{k+1/2}-z^{k}}^{2}.
\end{split}
\end{equation}
By the strong monotonicity of $\opF_1$, we obtain 
\begin{align}
    \inner{\opV_{k}(z^{k+1/2}),z^{k+1/2}-z}&=\inner{\opF_{2}(z^{k+1/2}),z^{k+1/2}-z}+\sigma_{k}\inner{\opF_{1}(z^{k+1/2}),z^{k+1/2}-z}\nonumber\\
    & \hspace{-5em}\geq \inner{\opF_{2}(z),z^{k+1/2}-z}+\sigma_{k}\inner{\opF_{1}(z),z^{k+1/2}-z}+\sigma_{k}\mu\norm{z^{k+1/2}-z}^{2}\nonumber\\
    &\hspace{-5em}\geq \inner{\opF_{2}(z),z^{k+1/2}-z}+\sigma_{k}\inner{\opF_{1}(z),z^{k+1/2}-z}\nonumber\\
    &\hspace{-4em}+\frac{\sigma_{k}\mu}{2}\norm{z^{k}-z}^{2} -\sigma_{k}\mu\norm{z^{k+1/2}-z^{k}}^{2},
    \label{eq:by_str_monotonicity}
\end{align}
where the last inequality follows from the triangle inequality and $(a+b)^2/2\leq a^2+b^2 \Leftrightarrow a^2 \geq (a+b)^2/2 - b^2$. 
Plugging \eqref{eq:SM_proof_4} and \eqref{eq:by_str_monotonicity} multiplied by $-1$ into \eqref{eq:SM_proof_1}, recalling $G_k=g_2+\sigma_k g_1$, we obtain
\begin{align}
E_{k+1}(z)&+D_{k+1}\nonumber\\
&\leq (1-t_{k}\sigma_{k}\mu)(E_{k}(z)+D_{k})+(4t^{2}_{k}L^{2}_{k}+t_{k}\mu\sigma_{k}-\frac{1}{2})\norm{z^{k+1/2}-z^{k}}^{2} \nonumber\\
&-t_{k}(\inner{\opF_{2}(z),z^{k+1/2}-z}+g_{2}(z^{k+1/2})-g_{2}(z))\nonumber\\
&-t_{k}\sigma_{k}(\inner{\opF_{1}(z),z^{k+1/2}-z}+g_{1}(z^{k+1/2})-g_{1}(z))\nonumber\\
&\leq (1-t_{k}\sigma_{k}\mu)(E_{k}(z)+D_{k})\nonumber\\
&-t_k \left( H^{(\opF_{2},g_{2})}(z^{k+1/2},z)+\sigma_{k}H^{(\opF_{1},g_{1})}(z^{k+1/2},z)\right), 
\label{eq:main_recursion_sc}
\end{align} 
where the last inequality uses \eqref{eq:StepSM} and \eqref{eq:HFG_def}. We now use this recursion to establish rates on the feasibility and optimality gap, starting with the feasibility gap. 

Since the analysis for the monotone setting holds also for strongly monotone case, we have that either Assumption \ref{ass:ACcondition} holds and implies by Lemma \ref{lem:boundsequence} that the sequences generated by the algorithm are bounded or $\dom(g_{1})\cap\dom(g_{2})$ is compact.
Thus, in any case, for a compact subset $\scrU_{2}\subset\dom(g_{1})\cap\dom(g_{2})$, there exists a constant $\tilde{C}_{\scrU_{2}}$ such that, for all $z \in \scrU_2, k\geq 1$, (cf. the derivation of \eqref{eq:H_1_bound})
\begin{equation*}
-H^{(\opF_{1},g_{1})}(z^{k+1/2},z)=\inner{\opF_{1}(z),z-z^{k+1/2}}+g_{1}(z)-g_1(z^{k+1/2})\leq \tilde{C}_{\scrU_2}.
\end{equation*}
Combining this with \eqref{eq:main_recursion_sc} and recalling the notation \eqref{eq:W_k_def}, we get 
\begin{align*}
t_{k} H^{(\opF_{2},g_{2})}(z^{k+1/2},z) &\leq  \left(1-t_k\sigma_k\mu\right)W_k(z)  - W_{k+1}(z) + t_{k}\sigma_{k} \widetilde{C}_{\scrU_2}.
\end{align*}
Multiplying both sides by $\sigma_k\gamma_k$ (see \eqref{eq:Ergodic_SM} for the definition of $\gamma_k$) and using that $\sigma_k\geq \sigma_{k+1}$ and $W_k(z)\geq 0$ for $k\geq 1$, we obtain
\begin{align*}
t_{k} \sigma_k\gamma_kH^{(\opF_{2},g_{2})}(z^{k+1/2},z) &\leq  \sigma_k\gamma_{k-1}W_k(z)  - \sigma_{k+1}\gamma_kW_{k+1}(z) + t_{k}\sigma_{k}^2\gamma_k \widetilde{C}_{\scrU_2}.
\end{align*}
Telescoping these inequalities and using the convexity of $g_2$ and $D_1=0$, $\gamma_0=1$, for the ergodic average \eqref{eq:Ergodic_SM}, we obtain 
\begin{align*}
 H^{(\opF_{2},g_{2})}(\bar{z}^{K},z)   &\leq   \frac{\sigma_1 E_1(z)+\widetilde{C}_{\scrU_2} \sum_{i=1}^Kt_{i}\sigma_{i}^2\gamma_i}{\sum_{i=1}^Kt_{i}\sigma_{i}\gamma_i}.
\end{align*}
Taking the supremum over $z \in \scrU_{2}$ on both sides and using the definition of the feasibility gap \eqref{eq:FeasGap}, we obtain the upper bound \eqref{eq:Feas_SM}.

To obtain a bound in terms of the optimality gap, consider an arbitrary reference point $z\in\scrS_2$ and rearrange \eqref{eq:main_recursion_sc} to arrive at
\begin{align*}
t_{k}\sigma_{k}H^{(\opF_{1},g_{1})}(z^{k+1/2},z) &\leq  \left(1-t_k\sigma_k\mu\right)(E_k(z)+D_k) - (E_{k+1}(z)+D_{k+1}).
\end{align*}
Multiplying the previous inequality by $\gamma_k$  yields
\begin{align*}
t_{k}\sigma_{k}\gamma_kH^{(\opF_{1},g_{1})}(z^{k+1/2},z) &\leq  \gamma_{k-1}(E_k(z)+D_k) - \gamma_k(E_{k+1}(z)+D_{k+1}).
\end{align*}
Telescoping these inequalities, using the convexity of $g_1$, $D_1=0$, $\gamma_0=1$, for the ergodic average \eqref{eq:Ergodic_SM}, we obtain 
\begin{align*}
 H^{(\opF_{1},g_{1})}(\bar{z}^{K},z) &\leq   \frac{E_1(z)}{\sum_{i=1}^Kt_{i}\sigma_{i}\gamma_i}.
\end{align*}
Since $\scrU_{1}\subset\dom(g_{1})\cap\dom(g_{2})$ is a nonempty compact set with $\scrU_{1}\cap\scrS_{1}\neq\emptyset$, we have that $\scrU_{1}\cap\scrS_{2}\neq\emptyset$. Taking supremum over $\scrU_{1}\cap\scrS_{2}$, we arrive at the upper bound in \eqref{eq:Opt_unbounded_SM}.
The lower bound in \eqref{eq:Opt_unbounded_SM} is the same as in  \eqref{eq:Opt_unbounded}. 

If additionally the lower level solution set enjoys $(\alpha,\rho)$-weak sharpness we obtain \eqref{eq:complexitySharp_SM} by combining \eqref{eq:Opt_unbounded_SM}, \eqref{eq:WS}, and \eqref{eq:Feas_SM}.
\end{proof}
To have a better understanding of the obtained rates, we now make a particular choice of the sequences $(t_{k})_{k\geq 1}$, $(\sigma_{k})_{k\geq 1}$.
\begin{corollary}\label{cor:tsigmaSM} 
Let the same Assumptions as in Theorem \ref{thm:ComplexitySM} be in place. Assume that the regularization sequences $(t_{k})_{k\geq 1}$ and $(\sigma_{k})_{k\geq 1}$ are chosen as 
\begin{equation}
    t_k=\frac{1}{4(L_k+\sigma_k\mu)}=\frac{1}{4(L_{\opF_2}+\sigma_k(L_{\opF_1}+\mu))}, \quad \sigma_k=\frac{4L_{\opF_2}}{\mu k}, \quad k\geq 1.
\end{equation}
Then, we have
\begin{align}
    \label{eq:lower_level_rate_unbounded_detlta_SM}
      &\Theta_{\rm Feas}(\bar{z}^{K}\vert\scrU_2)\leq \frac{8L_{\opF_2}(L_{\opF_1}+\mu)\sup_{z\in\scrU_2}\norm{z^{1}-z}^{2}}{\mu K}+\tilde{C}_{\scrU_{2}} \frac{4L_{\opF_2}(1+\ln K) }{ \mu K },\\
        \label{eq:upper_level_rate_unbounded_detlta_SM}
 -B_{\scrU_{1}}&\left[\frac{8L_{\opF_2}(L_{\opF_1}+\mu)\sup_{z\in\scrU_2}\norm{z^{1}-z}^{2}}{\mu K (\alpha/\rho)}+\tilde{C}_{\scrU_{2}} \frac{4L_{\opF_2}(1+\ln K) }{ \mu K (\alpha/\rho)} \right]^{1/\rho}\\
&\overset{(*)}{\leq}   \Theta_{\rm Opt}(\bar{z}^{K}\vert\scrU_1\cap\scrS_{2})\leq \frac{ 2(L_{\opF_1}+\mu)\sup_{z\in\scrU_{1}\cap\scrS_{2}}\norm{z^{1}-z}^{2}}{K},\nonumber
\end{align}
where the inequality $(*)$ holds under the $(\alpha,\rho)$-weak sharpness assumption.
\end{corollary} 
\begin{proof}
First, it is easy to check that \eqref{eq:StepSM} is indeed satisfied.
Further, $t_i\sigma_i=\frac{1}{L_{\opF_2}+\frac{4L_{\opF_2}}{\mu i}(L_{\opF_1}+\mu)}\cdot\frac{L_{\opF_2}}{\mu i}=\frac{1}{i+\kappa}\cdot\frac{1}{\mu }$,
where we denoted $\kappa=\frac{4}{\mu }(L_{\opF_1}+\mu)$. 
This gives us
\begin{align*}
\prod_{i=1}^k(1-t_i\sigma_i\mu)=
\prod_{i=1}^k \left(\frac{i+\kappa-1}{i+\kappa} \right)= \frac{\kappa}{k+\kappa}
\end{align*}
and thus $\gamma_k=1/\prod_{i=1}^k(1-t_i\sigma_i\mu)=\frac{k+\kappa}{\kappa}$.
Finally, we have
\begin{align*}
\sum_{i=1}^Kt_{i}\sigma_{i}\gamma_i &= \sum_{i=1}^K \frac{1}{i+\kappa}\cdot\frac{1}{\mu } \cdot \frac{i+\kappa}{\kappa} = \frac{K}{\mu\kappa}=\frac{K}{4(L_{\opF_1}+\mu)},\\
\sum_{i=1}^Kt_{i}\sigma_{i}^2\gamma_i &= \sum_{i=1}^K \frac{1}{i+\kappa}\cdot\frac{1}{\mu } \cdot \frac{4L_{\opF_2}}{\mu i} \cdot \frac{i+\kappa}{\kappa} = \frac{L_{\opF_2}\sum_{i=1}^K \frac{1}{i}}{\mu(L_{\opF_1}+\mu)}  \leq \frac{L_{\opF_2}(1+\ln K)}{\mu(L_{\opF_1}+\mu)}.
\end{align*}
Substituting this into \eqref{eq:Feas_SM}, we obtain  \eqref{eq:lower_level_rate_unbounded_detlta_SM}. Substituting this into \eqref{eq:Opt_unbounded_SM} and \eqref{eq:complexitySharp_SM}, we obtain \eqref{eq:upper_level_rate_unbounded_detlta_SM}.
\end{proof}

\subsection{Forward-backward-forward splitting}
The optimistic extragradient method saves one evaluation of the time-varying operator $\opV_{k}$ per iteration compared to the standard extragradient method. Yet, it still requires two evaluations of the proximal operator of $G_{k}$ per iteration. An attractive modified extragradient scheme, which uses past values of the operator but saves on one evaluation of the proximal operator, is the following optimistic version of Tseng's forward-backward-forward method \cite{Tsen00} 
\begin{align*}
z^{k+1/2}&=\prox_{t_{k}G_{k}}(z^{k}-t_{k}\opV_{k}(z^{k-1/2})),\\ 
z^{k+1}&=z^{k+1/2}-t_k(\opV_{k}(z^{k+1/2})-\opV_{k}(z^{k-1/2})).
\end{align*}
By construction, there exits a $\xi^{k}\in\partial G_k(z^{k+1/2})$ such that
\begin{align*}
  z^{k+1/2}=z^k-t_k(\opV_{k}(z^{k-1/2})+\xi^k),\text{ and } 
  z^{k+1}=
  z^k-t_k(\opV_k(z^{k+1/2})+\xi^k).
\end{align*}
Eq. \eqref{eq:zk2} gives 
\begin{align}
    \norm{z^{k+1}-z}^2
    &=\norm{z^k-z}^2-\norm{z^{k+1}-z^k}^2+2\inner{z^{k+1}-z^k,z^{k+1}-z}\nonumber\\
    &\hspace{-5em}=\norm{z^k-z}^2-\norm{z^{k+1}-z^k}^2 +\inner{z^k-t_k(\opV_k(z^{k+1/2})+\xi^k)-z^k,z^{k+1}-z}\nonumber\\
    &\hspace{-5em}=\norm{z^k-z}^2-\norm{z^{k+1}-z^k}^2-2t_k\inner{\opV_k(z^{k+1/2})+\xi^k,z^{k+1}-z^{k+1/2}}\nonumber\\
    &\hspace{-3em}-2t_k\inner{\opV_k(z^{k+1/2})+\xi^k,z^{k+1/2}-z}.\label{FBFeq1}
\end{align}
Now observe that $t_k(\opV_k(z^{k+1/2})+\xi^k)=z^k-z^{k+1}=z^k-z^{k+1/2}+t_k(\opV_k(z^{k+1/2})-\opV_k(z^{k-1/2})).$ As a consequence, we get 
\begin{align*}
   &2t_k\inner{\opV_k(z^{k+1/2})+\xi^k,z^{k+1}-z^{k+1/2}}=2\inner{z^k-z^{k+1/2},z^{k+1}-z^{k+1/2}}\\
   &+2t_k\inner{\opV_k(z^{k+1/2})-\opV_k(z^{k-1/2}),z^{k+1}-z^{k+1/2}}\\
   &=\norm{z^k-z^{k+1/2}}^2+\norm{z^{k+1}-z^{k+1/2}}^2-\norm{z^{k+1}-z^k}^2\\
   &+2t_k\inner{\opV_k(z^{k+1/2})-\opV_k(z^{k-1/2}),z^{k+1}-z^{k+1/2}}.
\end{align*}
Combining this estimate with \eqref{FBFeq1} gives
\begin{align*}
  &\norm{z^{k+1}-z}^2=\norm{z^k-z}^2-\norm{z^k-z^{k+1/2}}^2-\norm{z^{k+1}-z^{k+1/2}}^2\\
  &+ 2t_k\inner{\opV_k(z^{k-1/2})-\opV_k(z^{k+1/2}),z^{k+1}-z^{k+1/2}}  -2t_k\inner{\opV_{k}(z^{k+1/2})+\xi_k,z^{k+1/2}-z}.
\end{align*}
Since $\xi^k\in\partial G_k(z^{k+1/2})$, the subgradient inequality allows us to continue with
\begin{align*}
    \norm{z^{k+1}-z}^{2}&\leq \norm{z^{k}-z}^{2}-\norm{z^{k+1/2}-z^{k}}^{2}-\norm{z^{k+1}-z^{k+1/2}}^{2}\\
    &-2t_{k}\Psi_{k}(z)-2t_{k}\inner{\opV_{k}(z^{k+1/2})-\opV_{k}(z^{k-1/2}),z^{k+1}-z^{k+1/2}},
\end{align*}
where we have used the notation introduced in \eqref{eq:shortcut_notation_Psi}. This is exactly the energy estimate \eqref{eq:energy_important_111}. 
Repeating the same steps as in the rest of Section \ref{sec:energy}, we obtain the same energy bound as reported in eq. \eqref{eq:Energy}. From here on, we can repeat the same arguments as in Section \ref{sec:AllProofs} in order to obtain the same results as reported in Section \ref{sec:mainresults} for the optimistic extragradient method.

\section{Numerical Experiments}
\label{sec:numerics}
%
We verify the performance of the optimistic extragradient method with two numerical examples. In subsection \ref{sec:HierarchyNE} we consider a simple instance of a hierarchical Nash equilibrium problem as described in Section \ref{ex:HierarchyNE}, and compare our Algorithm with PASTA from \cite{LamSagSIOPT25}. Subsection \ref{sec:DUE} described a typical test-instance in selecting a dynamic user equilibrium with elastic demand \cite{han2015elastic}. 

\subsection{Hierarchical Nash equilibria}\label{sec:HierarchyNE}
To provide a first illustration on the practical performance of our method, we've tested Algorithm \ref{alg:Alg1} on a simple version of the hierarchical equilibrium problem described in Section \ref{sec:HierarchyNE}, taken from \cite{LamSagSIOPT25}. We've chosen a deliberately simple example for which the equilibrium set of the game can be computed explicitly and we can monitor the evolution of the inexactness of the algorithm. The example allows us to perform a numerical comparison with the PASTA algorithm employed in \cite{LamSagSIOPT25}. 

\paragraph{Design of the example.}
We consider $N= 4$ lower-level players and $M= 2$ upper-level players, with $x^1 = (y^2,y^4), x^2 = (y^1,y^3)$. This models a principal-agent relationship: Player 1 is the principal of agents 2 and 4, while player 2 is the principal of agent 1 and 3. The goal of the principals is to pick their favorite Nash equilibrium. The player-cost functions are 
\begin{align*}
&h^{\ell}_{1}(y^{1},y^{-1})= 0.5(y^1)^2 + y^1(y^2 +2y^3+y^4-100),\;\varphi^{\ell}_{1}(y^{1})=\iota_{[-100,50]}(y^{1}),\\
&h^{\ell}_{2}(y^{2},y^{-2})=0.5(y^{2})^{2}+y^{2}(y^{1}+y^{3}+y^{4}-50),\\
&\varphi^{\ell}_{2}(y^{2})=\max\{-10(y^{2}-15),0\}+\iota_{[0,50]}(y^{2}),\\
&h^{\ell}_{3}(y^{3},y^{-3})=0.5(y^{3})^{2}+y^{3}(y^{2}+y^{4}-100),\;\varphi^{\ell}_{3}(y^{3})=\iota_{[0,100]}(y^{3})\\
&h^{\ell}_{4}(y^{4},y^{-4})=0.5(y^{4})^{2}+y^{4}(y^{1}+y^{2}+y^{3}-50),\;\varphi^{\ell}_{4}(y^{4})=\iota_{[0,50]}(y^{4}).
\end{align*}
For the upper-level players, we assume 
\begin{align*}
    &h^{u}_{1}(x_{1},x_{2})=(y^{2}-20)^{2}+(y^{4}-50)^{2}+(y^{2}+y^{4})(y^{1}+y^{3}),\varphi^{u}_{1}(x^{1})=0,\\
    &h^{u}_{2}(x_{1},x_{2})=(y^{1})^{2}+y^{1}(y^{2}+y^{3})+(y^{3})^{2}+y^{3}(y^{2}+y^{4}),\varphi^{u}_{2}(x^{2})=0.
\end{align*}
Following the identification of the problem data as described in Example \ref{ex:HierarchyNE}, we see that $\dom(g_{1})\cap \dom(g_{2})$ is a compact box in $\R^{4}$ denoted by $Y$. When evaluating the feasibility gap, we can thus use $\scrU_{2}=Y$ as the localization. A special feature of this example is that we can obtain an explicit expression for the lower-level equilibrium set: $\scrS_{2}=\{(-50,y^2,50,50 - y^2) : 15 \leq y^2 \leq 50\}$. Thus, the unique variational equilibrium of this game is $(-50,15,50,35)$. We solved this game problem with Algorithm \ref{alg:Alg1} using the regularization sequence $\sigma_{k}=\frac{1}{(k+3)^{1/2}}$ and constant step size $t_k=t$ according to the theory. Figure \ref{fig:GNE} displays the temporal evolution of the strategies of the four players, as well as the distance to the unique variational equilibrium.

\begin{figure}[t]
    \centering
    \includegraphics[width=0.45\linewidth]{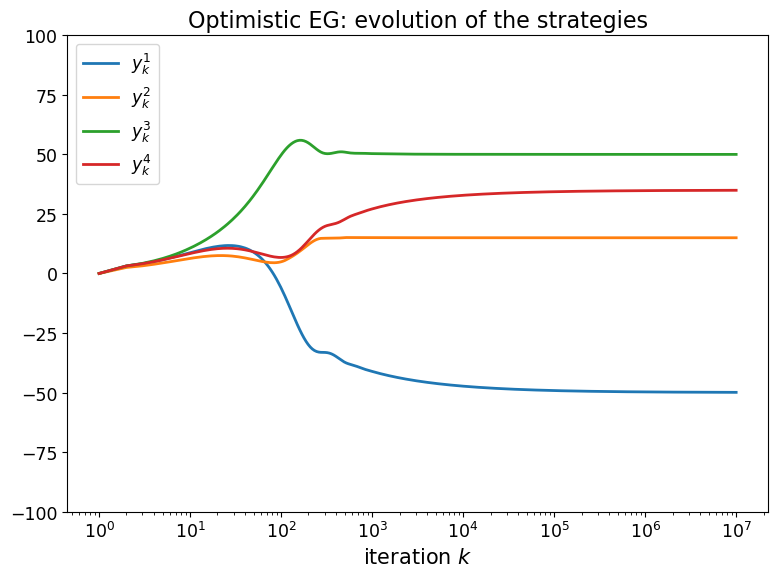}    \includegraphics[width=0.45\linewidth]{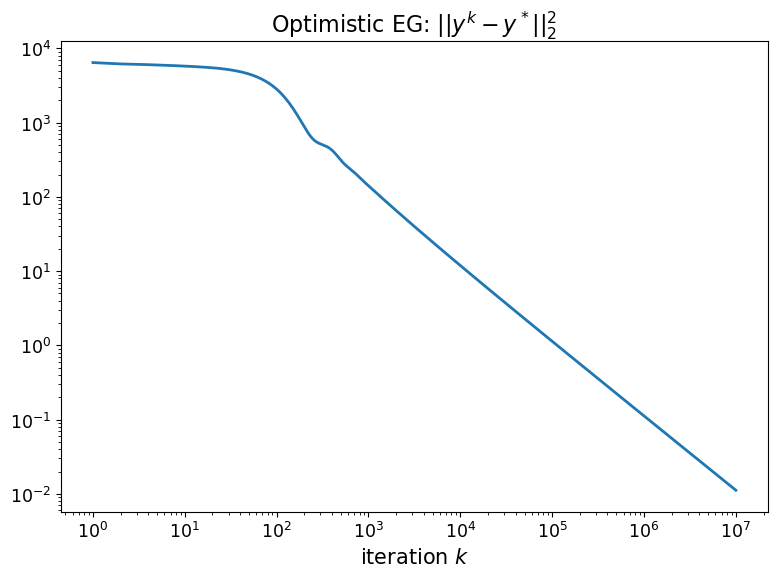}
    \caption{Evolution of the strategy profile and error plot for the hierarchical GNEP.}
    \label{fig:GNE}
\end{figure}

\paragraph{Comparison with PASTA.}
To perform numerical comparisons, we solve the same instance with PASTA, Algorithm~7.1 of \cite{LamSagSIOPT25}. PASTA is a projected-subgradient scheme with a vanishing Tikhonov parameter. It therefore requires only a subgradient oracle for $\opF_1,\opF_2$, which therefore absorb all potential non-smooth terms in the problem formulation. Specifically, it handles the nonsmooth term $\varphi^\ell_2$ through a subgradient selection rather than a proximal step, and it requires only a single projection per iteration, while Algorithm~\ref{alg:Alg1} relies on two proximal steps. Both methods are initiated from $y^1=(0,0,0,0)$ for $K=10^7$ iterations, PASTA with the exponent schedule of \cite[Prop.~7.1]{LamSagSIOPT25} and Algorithm~\ref{alg:Alg1} with $\sigma_k,t_k$ as above; we write $y^k$ and $y^k_P$ for the respective raw iterates.

Figure~\ref{fig:PASTAvsEG-error} compares $\|\cdot-y^*\|_\infty$ for both methods' iterates and averaged sequences. PASTA is faster up to $k\approx1.7\times10^5$, but its error visibly flattens thereafter, plateauing near $0.23$ even at iteration count $k=10^7$. Algorithm~\ref{alg:Alg1} iterate overtakes PASTA permanently beyond that point and keeps decreasing, reaching $\|y^{K}-y^*\|_\infty\approx0.041$, roughly $5.6\times$ smaller than PASTA's result after the same number of iterations.

\begin{figure}[t]
    \centering
    \includegraphics[width=0.65\linewidth]{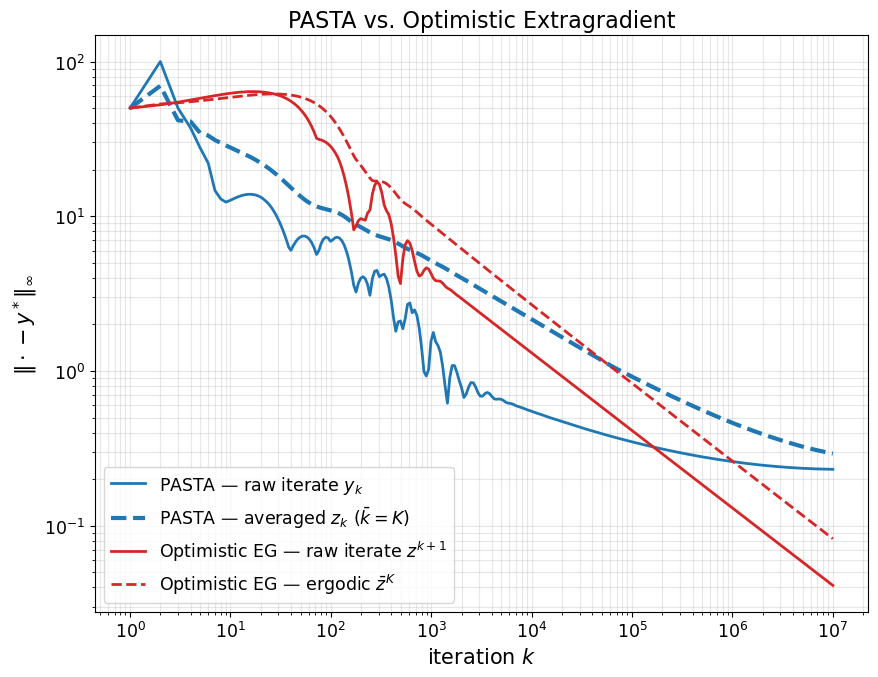}
    \caption{$\|\cdot-y^*\|_\infty$ for PASTA and Algorithm~\ref{alg:Alg1}, raw iterates and averaged sequences.}
    \label{fig:PASTAvsEG-error}
\end{figure}

The distance to $y^*$ is only computable here because the equilibrium happens to be known in closed form. To illustrate the tightness of the theoretical rates in terms of the feasibility and optimality gap, we compare the evolution of $\Theta_{\rm Feas}$ and $\Theta_{\rm Opt}$  as well. The obtained results are displayed in Figure~\ref{fig:PASTAvsEG-gaps}. At $k=10^7$, $\Theta_{\rm Feas}=1.30$ for Algorithm~\ref{alg:Alg1} against PASTA's $11.97$, and $\Theta_{\rm Opt}=-0.29$ against $-1.73$ (Table~\ref{tab:PASTAvsEG}), both quantities are uniformly smaller for Algorithm~\ref{alg:Alg1} from $k\approx10^5$ onward.

\begin{figure}[t]
    \centering
    \includegraphics[width=\linewidth]{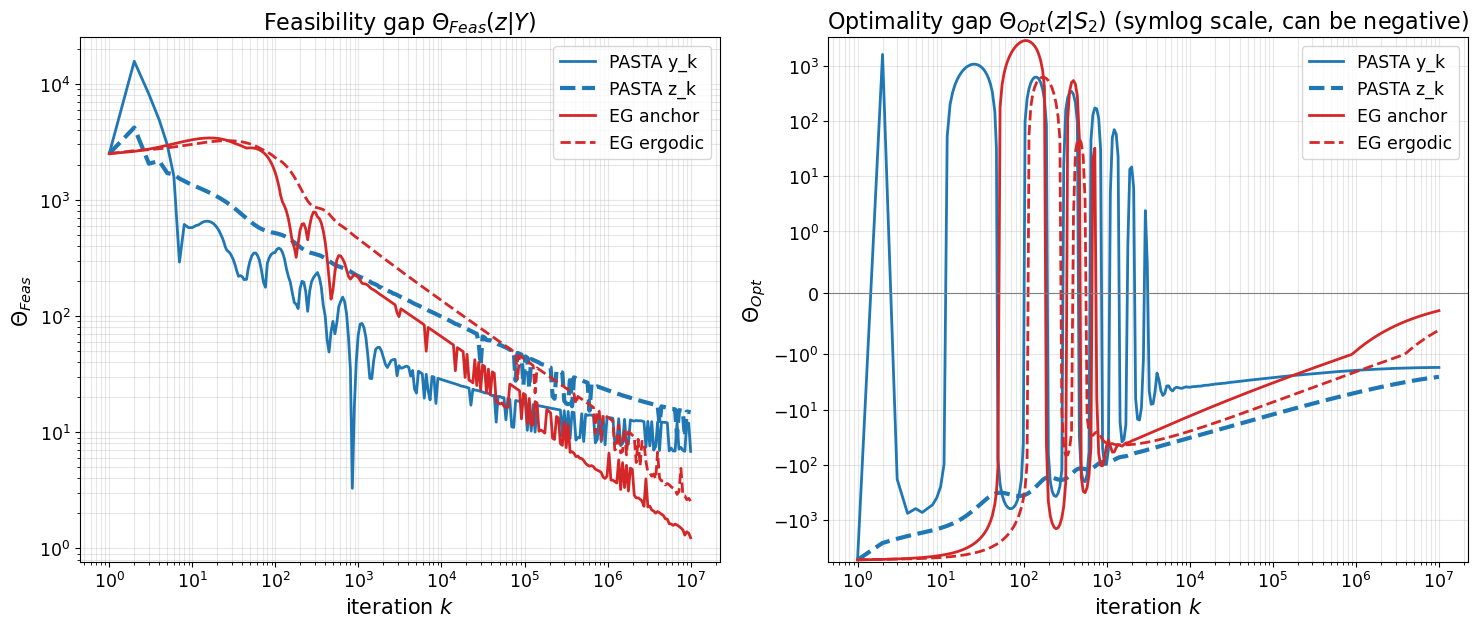}
    \caption{Feasibility gap $\Theta_{\rm Feas}$ and optimality gap $\Theta_{\rm Opt}$ for PASTA and Algorithm~\ref{alg:Alg1}, same run as Figure~\ref{fig:PASTAvsEG-error}. $\Theta_{\rm Opt}$ can
    be negative and is shown on a symmetric log scale.}
    \label{fig:PASTAvsEG-gaps}
\end{figure}

\begin{table}[t]
\centering
\caption{Raw-iterate comparison of PASTA and Algorithm~\ref{alg:Alg1} on the game.}
\label{tab:PASTAvsEG}
\begin{tabular}{c cc cc cc}
\toprule
 & \multicolumn{2}{c}{$\|\cdot-y^*\|_\infty$} & \multicolumn{2}{c}{$\Theta_{\rm Feas}$} & \multicolumn{2}{c}{$\Theta_{\rm Opt}$} \\
\cmidrule(lr){2-3}\cmidrule(lr){4-5}\cmidrule(lr){6-7}
$k$ & PASTA & Alg.~\ref{alg:Alg1} & PASTA & Alg.~\ref{alg:Alg1} & PASTA & Alg.~\ref{alg:Alg1} \\
\midrule
$10^3$  & 1.612 & 4.223 & 71.67 & 214.27 & $-94.48$ & $-42.68$ \\
$10^4$  & 0.549 & 1.309 & 18.44 & 66.46  & $-3.95$  & $-13.10$ \\
$10^5$  & 0.348 & 0.412 & 17.99 & 12.47  & $-2.64$  & $-3.30$  \\
$10^6$  & 0.260 & 0.130 & 13.29 & 4.14   & $-1.95$  & $-0.95$  \\
$10^7$  & 0.231 & 0.041 & 11.97 & 1.30   & $-1.73$  & $-0.29$  \\
\bottomrule
\end{tabular}
\end{table}

\subsection{Dynamic User equilibrium with elastic demand}
\label{sec:DUE}
Dynamic user equilibrium (DUE) is a Nash-like solution concept describing an equilibrium in dynamic traffic systems over a fixed planning period. Most of the studies of DUE reported in the dynamic traffic assignment literature are about dynamic user equilibrium
with constant travel demand for each origin–destination (O/D) pair. It is, of course, not generally true that travel demand is fixed,
even for short time horizons. To determine demand profiles together with equilibrium flows, the important extension of DUE with elastic demand has been studied in the literature \cite{han2015elastic}. Since the dynamic user equilibrium includes physical flow propagation together with behavioral principles, the question of equilibrium selection becomes relevant.
\paragraph{Mathematical formulation}
Let $G=(V,A)$ be the directed graph with $V$ the set of nodes representing traffic intersections (junctions) and $A$ the set of arcs, representing road segments. A path $p$ in the graph $G$ is identified with a non-repeating finite sequence of links it traverses. The set of paths in $G$ is denoted by $\scrP$. 
The planning horizon is the compact interval $I\eqdef [t_{0},t_{f}]\subset[0,\infty).$ The path departure rate at time $t$ (in vehicles per unit time) is measured by a function $h_{p}:[t_{0},t_{f}]\to[0,\infty)$. The positive number $h_{p}(t)$ is interpreted as a path flow measured at the entrance of the first link of the relevant path. As is common in the literature, we assume that the path departure times are square-integrable functions over the time domain $I$. The profile of path departure rates is $h\eqdef(h_{p})_{p\in\scrP}\in (L^{2}([t_{0},t_{f}];[0,\infty)))^{\abs{\scrP}}.$ We introduce the trip table $Q\eqdef (Q_{ij})_{(i,j)\in\scrW}$, where $Q_{ij}$ is the demand of traffic between O/D pair $(i,j)\in\scrW\subset V\times V$ and $\scrW$ is the set of origin-destination pairs. We assume that for every O/D pair $w\in\scrW$, there exists a finite constant $Q_{w}^{\max}>0$ such that $0\leq Q_{w}\leq Q_{w}^{\max}$. Flow conservation gives the constraint 
$$
\sum_{p\in\scrP_{ij}}\int_{I}h_{p}(t)\dif t=Q_{ij} \qquad\forall (i,j)\in\scrW,
$$
where $\scrP_{ij}\subset\scrP$ is the set of paths in $G$ connecting the O/D pair $(i,j)$. In the elastic demand setting, demand is not exogenous, but instead part of the equilibrium. The feasible set of path-departure rates and consistent demands is the set 
$$
\Lambda=\{(h,Q)\in\scrZ\vert \sum_{p\in \scrP_{ij}}\int_{I}h_{p}(t)\dif t=Q_{ij}\text{ and }Q_{ij}\in[0,Q_{ij}^{\max}]\quad \forall (i,j)\in\scrW\}.
$$
Let $\Delta$ denote the arc-path incidence matrix. 
The instantaneous link-flow vector is $x(t)=\Delta h(t)$, so that $x_{a}(t)=\sum_{p\in\scrP:a\in A(p)}h_{p}(t)$, where $A(p)\subset A$ is the set of arcs incident to path $p$. 
\paragraph{The lower level VI.}
Each link has the increasing affine cost
\begin{equation}
  \tau_a(x_a)=t_a^0+\kappa_a x_a,
  \qquad \kappa_a=0.15\,t_a^0/C_a>0,
\label{eq:linkcost}
\end{equation}
where $t_a^0$ and $C_a$ are the benchmark free-flow time and capacity on arc $a$.  Let
$s_p(t)$ be a flow-independent schedule penalty, common to the paths of a
given O/D pair if desired.  The effective path-cost operator is computed as
\begin{equation}
  \Psi(h)(t)=\Delta^{\top}
  \bigl(t^0+K\Delta h(t)\bigr)+s(t),
  \qquad K=\operatorname{diag}(\kappa_a),
\label{eq:pathcost}
\end{equation}
where $t^{0}=(t^{0}_{a})_{a\in A}$ is the collection of free-flow times on each link, and $s(t)=(s_{p}(t))_{p\in \scrP}$ is a flow-independent path-delay penalty term. For each O/D pair, we use an affine decreasing inverse demand $d_{w}:\R^{\abs{\scrW}}_{+}\to[0,\infty)$
\begin{equation}
  d_w((Q_{w'})_{w'\in \scrW})=b_w-\eta_w Q_w,\qquad \eta_w>0.
\label{eq:demand}
\end{equation}
This modeling assumption means that the demand on different O/D pairs are independent. This is a significant simplification which allows us to give a straightforward monotonicity and Lipschitz continuity proof of the resulting variational inequality. 
The demand vector is $d(Q)=(d_{w}(Q))_{w\in\scrW}:\R^{\abs{\scrW}}_{+}\to\R_{+}^{\abs{\scrW}}$.

A solution to the DUE problem with elastic demand can be expressed in terms of the infinite-dimensional variational inequality \cite{han2015elastic}
\begin{equation*}
    \begin{split}
&\text{find }(h^{*},Q^{*})\in\Lambda\text{ s.t. }\\
&\sum_{p\in\scrP}\int_{t_{0}}^{t_{f}}\Psi_{p}(h^{*})(t)(h_{p}(t)-h^{*}_{p}(t))\dif t-\sum_{(i,j)\in\scrW}d_{ij}(Q^{*})(Q_{ij}-Q_{ij}^{*})\geq 0 \quad\forall (h,Q)\in\Lambda.
    \end{split}
\end{equation*}
To bring this into the compact form corresponding to the main problem formulation in this paper, we make use of the following inner product on the product space $\scrZ\eqdef (L^{2}([t_{0},t_{f}];\R_{+}))^{\abs{\scrP}}\times\R_{+}^{\abs{\scrW}}$. For $z_{1}=(h^{1},Q^{1}),z_{2}=(h^{2},Q^{2})\in\scrZ$ we define
\[
\inner{z_{1},z_{2}}\eqdef\sum_{p\in\scrP}\int_{I}h^{1}_{p}(t)h^{2}_{p}(t)\dif t+\sum_{(i,j)\in\scrW}Q^{1}_{ij}Q^{2}_{ij}. 
\]
The norm induced is the canonical one $\norm{z}\eqdef\sqrt{\inner{z,z}}.$ $(\scrZ,\inner{\cdot,\cdot})$ is a real Hilbert space.
In terms of this topological structure, we define the problem data
\begin{equation}
 \opF_2(h,Q)\eqdef
 \begin{pmatrix}
  \Psi(h)\\ -d(Q)
 \end{pmatrix},\quad g_{2}(h,Q)\eqdef \delta_{\Lambda}(h,Q).
\label{eq:F2}
\end{equation}
Thus, the lower-level variational inequality reads as
\begin{equation}
 \text{find }z^*\in\scrZ\;\text{ such that }
 \inner{\opF_2(z^*),z-z^*}+g_{2}(z^{*})-g_{2}(z)\geq0
 \quad\forall z\in\scrZ.
\label{eq:lowerVI}
\end{equation}
We verify the the problem satisfies the assumptions imposed on the data. 
\begin{lemma}
The operator $\opF_{2}:\scrZ\to\scrZ$ is monotone and Lipschitz in the above described topology.
\end{lemma}
\begin{proof}
For $z=(h,Q)$ and $\tilde z=(\tilde h,\tilde Q)$,
\begin{align*}
 \inner{\opF_2(z)-\opF_2(\tilde z),z-\tilde z}
 &=\int_{t_{0}}^{t_{f}}
 \norm{K^{1/2}\Delta(h(t)-\tilde h(t))}_2^2\,dt\\
 &\quad+(Q-\tilde Q)^{\top}
 \operatorname{diag}(\eta_w)(Q-\tilde Q)\geq0.
\end{align*}
Thus $\opF_2$ is monotone.  It is affine and therefore globally Lipschitz, with
\[
 L_{\opF_2}=\max\left\{
 \norm{\Delta^{\top}K\Delta}_2,
 \max_{w\in\scrW}\eta_w
 \right\}.
\]
\end{proof}
Although the model has a quadratic structure, we point out that the path delay operator is generally not strongly monotone in path-flow space because
$\ker(\Delta)$ can contain non-zero route-redistribution directions.  This is the source of the lower-level path-flow non-uniqueness used by the hierarchy. This makes the question of picking a favorable solution from the lower level solution set $\scrS_{2}$ a relevant question. 

\paragraph{Minimum-$L^2$ hierarchical selection.} 
To select among the set of user equilibria, we introduce a path-flow $L^2$ criterion
\begin{equation}
 J(h,Q)\eqdef\frac12\sum_{p\in\scrP}
 \int_{t_{0}}^{t_{f}} \abs{h_p(t)}^2\,dt.
\label{eq:upperobjective}
\end{equation}
The hierarchical optimization problem we consider is
\begin{equation}
 \min\left\{J(z):z\in\scrS_{2}\right\}.
\label{eq:hierarchy}
\end{equation}
It induces the upper-level data
\begin{equation}
 \opF_1(h,Q)=\nabla J(h,Q)=
 \begin{pmatrix}h\\0\end{pmatrix},\quad g_{1}=0
\label{eq:F1}
\end{equation}
Clearly, $\opF_{1}:\scrZ\to\scrZ$ is monotone and $1$-Lipschitz. Unlike a link-flow
norm, it can distinguish equilibria that have identical link flows but different path decompositions.

\begin{remark}
    The model presented above is infinite dimensional before discretization and uses a standard traffic network, but is a very simplified version of realistic dynamic user equilibrium models. It does not propagate vehicles or queues. This restriction
is essential to the transparent monotonicity proof. For realistic LWR, link-transmission, spillback, or generalized Vickrey network loading, continuity results are available but useful monotonicity estimates are generally unavailable \cite{thong2021computing}. Such models therefore fall outside our assumptions. 
\end{remark}

\paragraph{Time discretization and projection}
In order to test the performance of Algorithm \ref{alg:Alg1} on concrete instances, we have to discretize the time domain (i.e. we "first discretize, then optimize"). We divide the planning horizon $I=[t_{0},t_{f}]$ into $N$ intervals of width $dt$. On each time interval $I_{n}=[t_{0}+(n-1)dt,t_{0}+ndt],n=1,\ldots,N$, we define (with an obvious abuse of notation) a discrete path flow $h_{p}=(h_{p,n})_{n=1}^{N}\in\R^{N}_{+}$. A tuple $h\in\R_+^{|\scrP|\times N}$ represents a discrete departure time vector. For each O/D pair $w\in\scrW$, the orthogonal projection of a candidate pair $(u,r)$ onto the discretized feasible set 
\begin{equation}\label{eq:FeasDiscrete}
\{(h,Q)\in\R^{\abs{\scrP}N}_{+}\times\R^{\abs{\scrW}}\vert dt\sum_{p\in\scrP_{ij}}\sum_{n=1}^{N}h_{p,n}=Q_{ij},Q_{ij}\in[0,Q^{\max}_{ij}]\quad(i,j)\in\scrW\}
\end{equation}
can be easily computed using the norm $dt\norm{h-u}_2^2+\norm{Q-r}_2^2$. Specifically, it has the form
\begin{equation}
 h_{p,n}=\max\{u_{p,n}-{\lambda_w},0\},\qquad
 Q_w=\Pi_{[0,Q_w^{\max}]}(r_w+\lambda_w),
\label{eq:projection}
\end{equation}
where the unique scalar $\lambda_w$ solves
\begin{equation}
  dt\sum_{p\in\scrP_w}\sum_{n=1}^N
 \max\{u_{p,n}-\lambda_w,0\}
 -\Pi_{[0,Q_w^{\max}]}(r_w+\lambda_w)=0.
\label{eq:root}
\end{equation}
The l.h.s. is continuous and non-increasing, so bisection gives an inexpensive
and accurate projection independently for each O/D pair.

\subsubsection{A concrete instance} 
As concrete numerical test, we let $G=(V,A)$ be the directed Nguyen--Dupuis network with 13 nodes and 19 links (cf. Fig. \ref{fig:Network}). The O/D pairs are described by the set $\scrW=\{(1,2),(1,3),(4,2),(4,3)\}.$ We assume reference demand for each O/D pair $\bar Q=(400,800,600,200)$, and set $Q^{\max}_{w}=1.5\times\bar{Q}_{w}.$ Enumerating all simple directed paths gives respectively $8$, $6$, $5$, and $6$ paths, hence $|\scrP|=25$.
\begin{figure}
    \centering
    \includegraphics[width=0.5\linewidth]{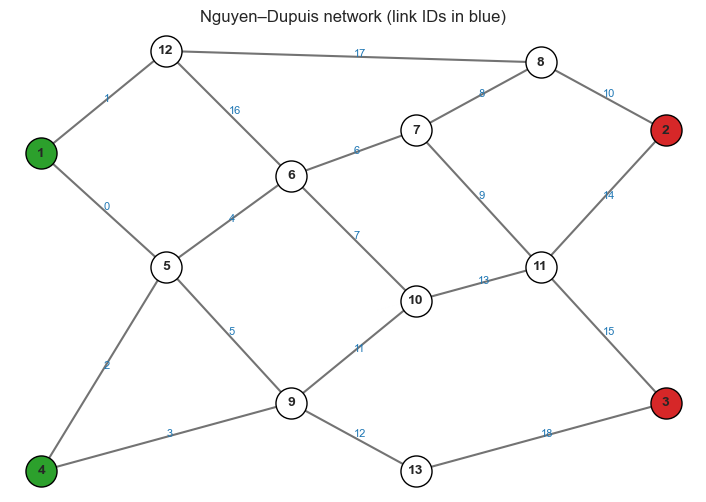}
    \caption{The Dupuis-Nguyen network.}
    \label{fig:Network}
\end{figure}
The rank of the incidence matrix $\Delta$ is 10. For the numerical implementation, we choose the time interval $I=[0,1],$ corresponding to a planning period of 1 hour. The discrete time grid is constructed by choosing $dt=1/N,N=48.$ In the experiments we are choosing the free-flow times 
$$
t^{0}=(
    12, 14, 14, 19, 5, 14, 8, 71, 8,
    14, 14, 12, 11, 13, 14, 13, 11, 5,12),
    $$
and arc capacities
\begin{align*}
C= &\left( 177, 104, 163, 235, 245, 121, 295, 213, 183, 291, 275, 221,\right. \\
    &\left. \quad  278, 241, 283, 169, 164, 179, 278\right). 
\end{align*}  
For these concrete numerical values, we perform various performance checks of our extragradient method (Algorithm \ref{alg:Alg1}).  As regularization sequence $(\sigma_{k})_{k}$ and step size $(t_{k})_{k}$ we make the same choices as in Section \ref{sec:HierarchyNE}. As in \cite{samadi_improved_2025}, we compare our method to Algorithm 1 in \cite{Facchinei:2014aa}, which
we also refer to by ISR-cvx\footnote{See \cite{cortild2025regularization} for a recent complexity analysis of this scheme for rather general hierarchical equilibrium problems.}. ISR-cvx is a double loop method, based on the auxiliary mapping 
\begin{equation}\label{eq:Aux}
\opF_2(z)+\eta_k\opF_1(z)+\alpha(z-z^k)+\NC_{\Lambda}(z),
\qquad \eta_k=\frac{\eta_0}{(k+1)^b}.
    \end{equation}
The outer loop consisting of updates of the regularization parameter $\eta_{k}$, while the inner loop performs forward-backward steps for fixed parameters $(\alpha,\eta_{k})$. Since the proximal regularization parameter $\alpha>0$ ensures strong monotonicity of the auxiliary mapping \eqref{eq:Aux}, the inner loop forward-backward steps converge linearly towards the unique fixed point of \eqref{eq:Aux}, denoted by $u^{k}\equiv \bar{u}(\eta_{k},\alpha,z^{k})$. We can therefor a-priori fix the number of inner loop steps, in order to get close to unique fixed point, and use this inexact version as the new anchor point $z^{k+1}$. The level of inexactness is denoted by $e_{k}$. Specifically, if $(z^{k,t})_{t}$ represents the sequence of the inner loop stopped at iteration $\scrT_{k}\in\N$, then $\norm{z^{k,\scrT_{k}}-u^{k}}\leq e_{k}.$ The ISR-cvx method is guaranteed to converge when $\sum_{k=0}^{\infty}\eta_{k}=\infty$ and $\lim_{k\to\infty}\frac{e_{k}}{\eta_{k}}=0.$ To meet these conditions, we use $\eta_{k}=\frac{\eta_{0}}{(k+1)^{b}}$ where $b\in(0,1]$ in ISR-cvx and terminate the inner-level loop after $\scrT_{k}=k+1$ number of iterations.

Our first comparison involves the time evolution of optimality and feasibility gaps. We first compute one highly accurate lower equilibrium $z^0$. Given the structure of the lower-level operator, another feasible point $z$ is a lower minimizer exactly when the quadratic-nullspace condition and the first-order equality hold. Here this reduces to fixed link flows at every time (i.e., $\Delta h(t)=\Delta h^0(t)$), fixed O/D demands (i.e. $Q=Q^0$), and $\langle \opF_2(z^0),z-z^0\rangle=0$. The code then minimizes $J(h)=\tfrac12\|h\|_{L^2}^2$ over that polyhedron, producing the value $J^{*}$. To measure the optimality gap we use the normalized objective function gap 
\begin{equation}\label{eq:OptimalityGap}
\frac{\abs{J(z^{k})-J^{*}}}{\abs{J^{*}}}. 
\end{equation}
We also monitor the evolution of the relative feasibility gap 
\begin{equation}\label{eq:FeasibilityGap}
\frac{\Theta_{\rm Feas}(z^k)}{\max\{1,\Theta_{\rm Feas}(z^{0})\}}.
\end{equation}
In order to illustrate the practical performance of Algorithm \ref{alg:Alg1} in solving the DUE problem with elastic demand, we let the method run until 10.000 lower level operator calls have been executed. By focusing on operator evaluation instead of algorithm iterations we get a cleaner comparison between the extragradient method and ISR-cvx. We compare the relative optimality and feasibility gaps generated by Algorithm \ref{alg:Alg1} and ISR-cvx for the proximal regularization parameter $\alpha=1$. The results are displayed in Fig. \ref{fig:Optimality}. While both methods display a similar pattern for reducing the relative gaps, our hierarchical extragradient method outperforms ISR-cvx significantly.
\begin{figure}
    \centering
    \includegraphics[width=0.95\textwidth]{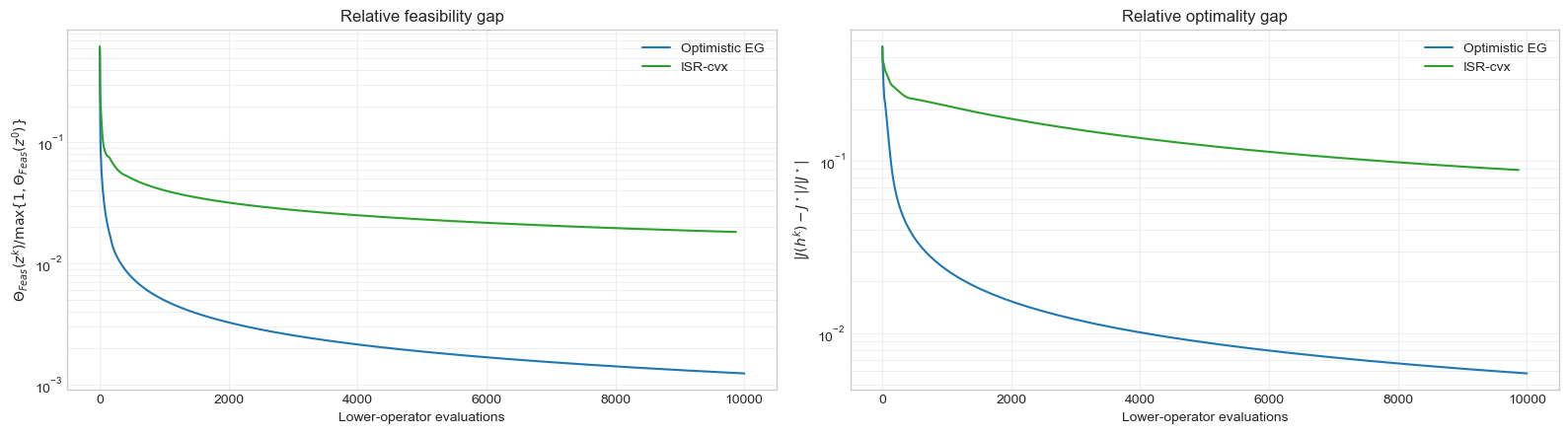}
    \caption{Relative optimality gap \eqref{eq:OptimalityGap} and feasibility gap \eqref{eq:FeasibilityGap} after 10.000 operator evaluations generated by Algorithm \ref{alg:Alg1} and ISR-cvx. ISR-cvx is run with proximal regularization parameter $\alpha=1$.}
    \label{fig:Optimality}
\end{figure}
In order to illustrate the dependence of the performance of ISR-cvx with respect to the proximal regularization $\alpha$ and the penalization sequence $\eta_{k}$, we plot the evolution of the relative feasibility and optimality gaps for an array of parameter constellations $\alpha\in\{1,5,10\}$ and 
$b\in\{0.2,0.4,0.6,0.8,1.0\}$, so that $\eta_k=\eta_0/(k+1)^b$. The results are displayed in Fig. \ref{fig:GAPS_DUE}. We again observe that our extragradient method consistently outperforms ISR-cvx.

Fig. \ref{fig:DUE_demand} displays the evolution of the demand profile and the induced average travel costs $v^{k}$ under our extragradient method. The latter being defined as 
\[
v_{w}^{k}\eqdef\frac{1}{Q^{k}_{w}}\sum_{p\in\scrP_{w}}\int_{t_{0}}^{t_{f}}\Psi_{p}(t,h^{k})h^{k}_{p}(t)\dif t\qquad \forall w\in\scrW. 
\]
The rational for this measure is the following: Until a solution has been found (as allowed by the prescribed
tolerance), the travel costs within the same O/D pair are not equalized; thus we use the averaged travel cost to demonstrate
the demand-cost relationship stipulated by the inverse demand function. It can be seen from Fig. \ref{fig:DUE_demand} that
both $Q^k_{w}$ and $v^k_{w}$ converge to a fixed value when the algorithms terminate, and they converge to a state that is consistent with
the inverse demand function.

\begin{figure}
    \centering
    \includegraphics[width=0.95\linewidth]{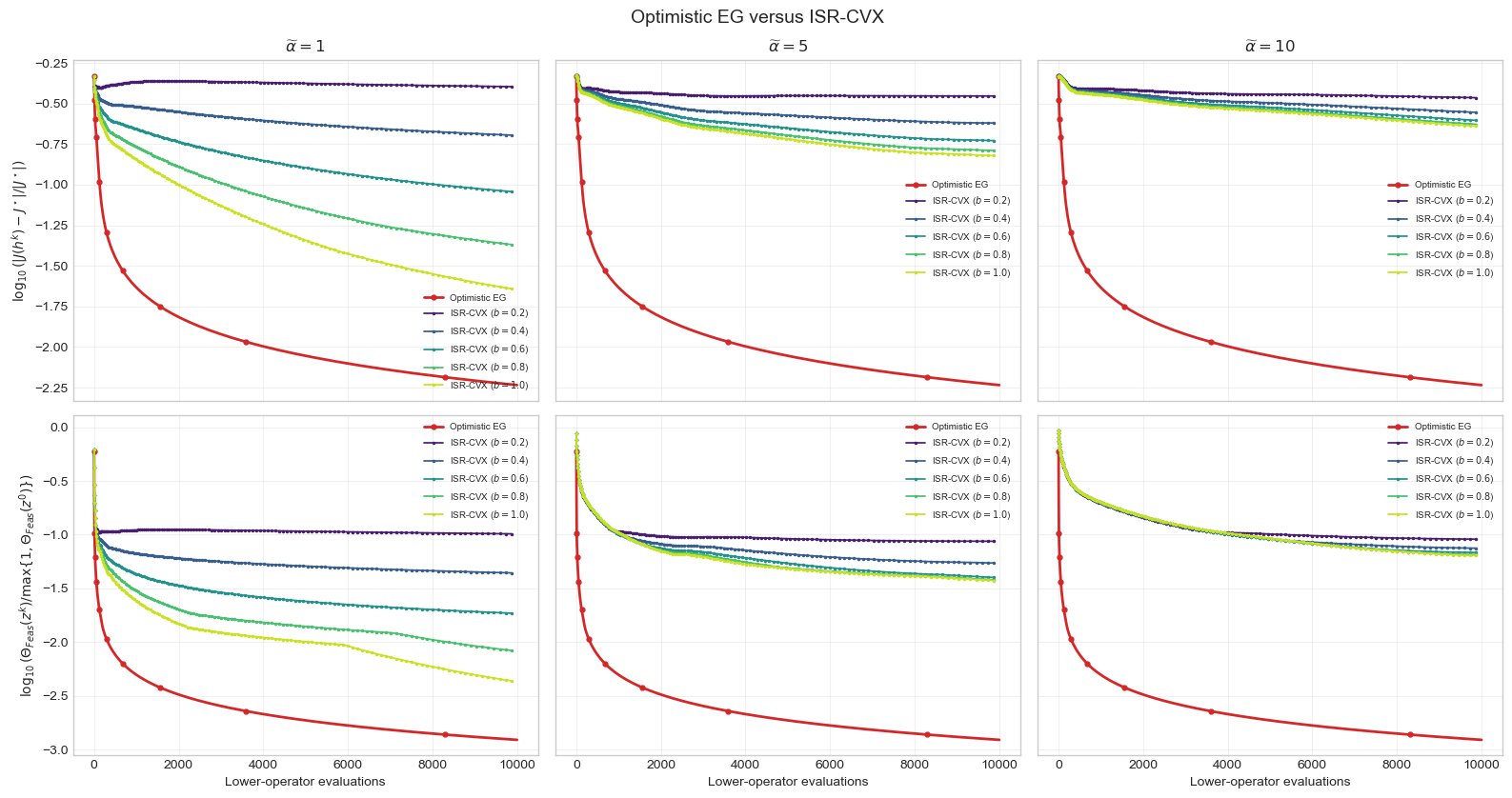}
    \caption{Feasibility and Optimality gap plots for Algorithm \ref{alg:Alg1} vs. ISR-cvx for solving the DUE problem with elastic demand.}
    \label{fig:GAPS_DUE}
\end{figure}

\begin{figure}
    \centering
    \includegraphics[width=0.95\linewidth]{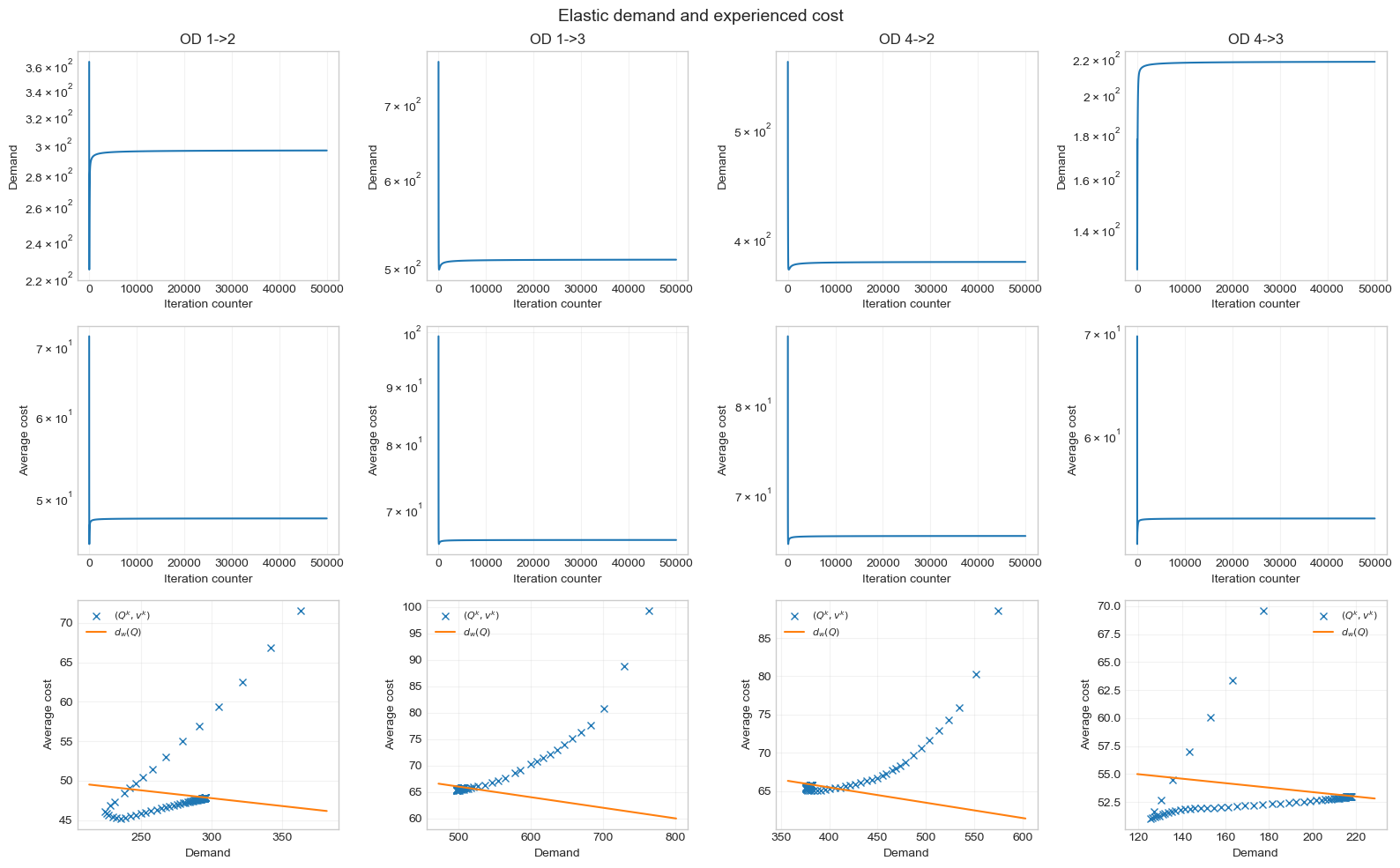}
    \caption{The elastic demand $Q^k$ and the average travel cost $v^k$ at each iteration. }
    \label{fig:DUE_demand}
\end{figure}

\section{Conclusion}
\label{sec:conclusion}
%

This paper proves several results on the convergence rates for solving hierarchical HVI's. Constructing a suitably defined optimistic extragradient method, we advance the state of the art along several dimensions. From a methodological perspective, we free-up existing rate analysis for hierarchical equilibrium problems from compactness assumptions. Following the classical summability conditions of Attouch and Czarnecki, we develop a proof technique establishing convergence of the iterates, and rates in terms of localized gap functions for assessing feasibility and optimality gaps. Future work will show the usefulness of our proof technique in related settings, including stochastic hierarchical equilibrium problems (see e.g.  \cite{dvurechensky2026stochastic}). Our analysis reveals close connections to geometric conditions imposed on the operator defining the lower level equilibrium problems, conditions which have so-far only been used in the context of hierarchical optimization. The research activities on this field are rather high and many interesting open questions remain. Future research should consider the development of accelerated methods, and adaptive variants of the existing scheme. We leave these questions as important extensions to future research.

\paragraph{Acknowledgements}
This research benefited from the support of the FMJH Program PGMO. MST's research is supported by the Deutsche Forschungsgemeinschaft (DFG) - Projektnummer 556222748 "non-stationary hierarchical minimization". PD's research is supported by the Deutsche Forschungsgemeinschaft (DFG, German Research Foundation) under Germany's Excellence Strategy – The Berlin Mathematics Research Center MATH+ (EXC-2046/1, project ID: 390685689). 

We thank Johannes-Carl Schnebel for his help on conducting the numerical experiments, and Enis Chenchene, Radu I. Bot, as well as David A. Hulett for fruitful discussions on this topic. 

\begin{appendix}
\section*{Appendix}
%

\section{Auxiliary facts}
\label{app:aux} 

\begin{lemma}[Lemma 5.31, \cite{BauCom16}]
\label{lem:RS}
Let $(a_{k})_{k\geq1},(b_{k})_{k\geq1},(c_{k})_{k\geq1}$ be nonnegative sequences such that $\sum_{k\geq 1}c_{k}<\infty$ and 
$
a_{k+1}\leq a_k-b_{k}+c_{k}.
$ 
Then, $\lim_{k\to\infty}a_{k}$ exists and $\sum_{k\geq 1}b_{k}<\infty$.
\end{lemma}
Let $(x_{n})_{n\geq1}$ be a sequence in $\scrZ$ and $(\lambda_{n})_{n\geq1}$ a sequence of positive numbers such that $\sum_{n\geq1}\lambda_{n}=+\infty$. Define the sequence of weighted averages $z_{n}=\frac{1}{\tau_{n}}\sum_{k=1}^{n}\lambda_{k}x_{k}$, where $\tau_{n}\eqdef\sum_{k=1}^{n}\lambda_{k}$.
The next fact is Lemma 2.3 in \cite{Attouch:2010aa}.
\begin{lemma}[Opial-Pasty]
\label{lem:Opial}
Let $\scrS$ be a nonempty subset of $\scrZ$ and assume that $\lim_{n\to\infty}\norm{x_{n}-x}$ exists for every $x\in\scrS$. If every weak cluster point of $(x_{n})_{n\geq1}$ (respectively $(z_{n})_{n\geq1}$) lies in $\scrS$, then $(x_{n})_{n\geq1}$ (respectively $(z_{n})_{n\geq1}$) converges weakly to an element of $\scrS$ as $n\to\infty$. 
\end{lemma}
\section{The Fitzpatrick function}\label{app:FP}
Let $\opM:\scrZ\to 2^{\scrZ}$ be a maximally monotone operator. The \emph{Fitzpatrick function} \cite{fitzpatrick1988representing,BauCom16} $\scrF_{\opM}:\scrZ\to(-\infty,+\infty]$,  associated with the operator $\opM$, is defined as 
\begin{equation*}
\scrF_{\opM}(x,u)\eqdef \sup_{(y,v)\in\gr(\opM)}\{\inner{x,v}+\inner{y,u}-\inner{y,v}\}.
\end{equation*}
Following \cite{Borwein:2016aa}, we define the gap function $\gap_{\opM}(x)\eqdef \scrF_{\opM}(x,0)$. $\gap_{\opM}$ is convex, and in fact the smallest translation invariant gap function associated with the monotone operator $\opM$ \cite[Theorem 3.1]{Borwein:2016aa}. Importantly, this gives the properties $\gap_{\opM}(x)\geq 0$ and $\gap_{\opM}(x)=0$ if and only if $x\in\zer(\opM)$. To make this concept concrete, observe that if $\opM=\opF+\NC_{\scrC}$, then the above definition of the gap function reduces to the well-known Auslender dual gap function \cite{FacPan03,auslender1976optimisation}
\begin{equation*}
\gap_{\opF+\NC_{\scrC}}(x)=\sup_{y\in\scrC}\inner{\opF(y),x-y}.
\end{equation*}
If $\opM=\opF+\partial g$ for a function $g\in\Gamma_{0}(\scrZ)$, we easily obtain 
\begin{equation*}
\gap_{\opF+\partial g}(x)\leq\sup_{y\in\dom(g)}\inner{\opF(y),x-y}+g(x)-g(y). 
\end{equation*}
Recall the definition \eqref{eq:HFG_def} of the bifunction $H^{(\opF,g)}$.
If $\opF:\scrZ\to\scrZ$ is monotone and continuous on $\dom(g)$, it is easy to verify that $H^{(\opF,g)}(x,y)\leq -H^{(\opF,g)}(y,x)$ for all $(x,y)\in \dom(g)\times\dom(g)$ (i.e. $H^{(\opF,g)}$ is a monotone bifunction \cite{iusem2011maximal}). In the structured setting $\HVI(\opF,g)$, we obtain the following bounds on the Fitzpatrick function, showing its close connection to the bifunction $H^{(\opF,g)}$ and the Auslender dual gap function. First, the convex subgradient inequality yields the relation
\begin{align*}
\scrF_{\opF+\partial g}(x,u)&=\sup_{y\in\dom(\partial g),\xi\in\partial g(y)}\{\inner{x-y,\opF(y)}+\inner{\xi,x-y}+\inner{y,u}\}\\
&\leq \sup_{y\in\dom(g)}\{\inner{\opF(y),x-y}+g(x)-g(y)+\inner{y,u}\}\\
&=\sup_{y\in\dom(g)} \{H^{(\opF,g)}(x,y)+\inner{y,u}\} \eqdef\varphi^{(\opF,g)}(x,u).
\end{align*}
The function $\varphi^{(\opF,g)}(x,u)$ is thus seen as the Fitzpatrick transform of $H^{(\opF,g)}$. Second, since the function $x\mapsto H^{(\opF,g)}(x,y)$ is convex, we see 
\begin{equation}\label{eq:dualphi}
\varphi^{(\opF,g)}(x,u)\leq\sup_{y\in\dom(g)}\{\inner{y,u}-H^{(\opF,g)}(y,x)\}=\left(H^{(\opF,g)}(\bullet,x)\right)^{\ast}(u). 
    \end{equation}
In particular, 
\begin{align}
&\varphi^{(\opF,g)}(x,u)\geq\inner{x,u} \qquad\forall x\in\dom(g), \label{eq:FPlower}\\ 
&\scrF_{\opF+\partial g}(x,0)\leq \sup_{y\in\dom(g)}H^{(\opF,g)}(x,y)=\varphi^{(\opF,g)}(x,0). \label{eq:FPgap}
\end{align}
Following this notation, we define the restricted dual gap function as in \eqref{eq:gap_Def_new}.

\section{Omitted proofs}\label{app:proofs}

\begin{proof}[Proof of Lemma \ref{lem:gap}]
The convexity of $x\mapsto \Theta(x\vert\opF,g,\scrC)$ follows from the fact that it is a supremum of convex functions. Since $\scrC$ is compact, this function is well-defined.

The lower bound $\Theta(x\vert\opF,g,\scrC)\geq 0$ is clear for $x\in\scrC$. Let's assume that $x\in\scrC$ is a solution to \eqref{eq:VI_general}. Then, by the monotonicity of $\opF$ we have 
\[
\inner{\opF(y),y-x}+g(y)-g(x) \geq \inner{\opF(x),y-x}+g(y)-g(x)\geq 0\qquad\forall y\in\scrZ.
\]
Hence,
$\Theta(x\vert\opF,g,\scrC)=\sup_{y\in\scrC}\{\inner{\opF(y),x-y}+g(x)-g(y)\}\leq 0$.
This, together with already proved inequality $\Theta(x\vert\opF,g,\scrC)\geq 0$ implies $\Theta(x\vert\opF,g,\scrC)= 0$.

Now assume $\Theta(x\vert\opF,g,\scrC)=0$ for $x\in\scrC$. Then, 
$\inner{\opF(x'),x-x'}+g(x)-g(x')\leq 0$ for all $ x'\in\scrC.$ 
Equivalently, $\inner{\opF(x'),x'-x}+g(x')-g(x)\geq 0 $ for all $ x'\in\scrC.$
This means $x\in\scrC$ is a Minty (weak) solution to the  $\HVI(\opF,g)$ over the set $\scrC$. Let $w\in\scrC$ be arbitrary and consider $v=tw+(1-t)x$ for $t\in[0,1]$. Since $\scrC$ is convex, we have $v\in\scrC$. It follows 
\begin{align*}
0&\leq \inner{\opF(v),v-x}+g(v)-g(x)\\
&=\inner{\opF(x+t(w-x)),t(w-x)}+g(tw+(1-t)x)-g(x)\\
&\leq t\inner{\opF(x+t(w-x)),w-x}+t(g(w)-g(x)) 
\end{align*}
Dividing both sides by $t$ and then letting $t\to 0^{+}$, the weak continuity of $\opF$ implies 
\begin{equation}
\label{eq:lem:gap_proof_1}
    0\leq \inner{\opF(x),w-x}+g(w)-g(x).
\end{equation}
Since $w\in\scrC$ has been chosen arbitrarily, we see that $x$ is a Stampacchia (strong) solution of the hemivariational inequality with the additional constraint $x\in\scrC$. 

We now show that if $\scrC\subset \dom(g)$, and there exists $\epsilon>0$ such that $\scrU\triangleq \ball(x,\epsilon)\cap \scrC=\ball(x,\epsilon)\cap \dom(g)$, then \eqref{eq:lem:gap_proof_1} holds also for $w \in \dom(g)$. To that end, assume to the contrary, there exits $z\in \dom(g)$ for which $\inner{\opF(x),z-x}+g(z)-g(x)<0.$
Then, by convexity of $\dom(g)$ and definition of $\scrU$ there exists $\lambda>0$ small enough for which $w=x+\lambda(z-x)\in \scrU$ and by convexity
\begin{equation*}
0\leq\inner{\opF(x),w-x}+g(w)-g(x)\leq \lambda(\inner{F(x),z-x}+g(z)-g(x))<0,
\end{equation*}
thus arriving at a contradiction. Hence, $\inner{F(x),z-x}+g(z)-g(x) \geq 0$ for any $z \in \dom(g)$ and moreover for any $z \in \scrZ$, which finishes the proof that $x$ solves \eqref{eq:VI_general}.
\end{proof}

\begin{proof}[Proof of Proposition \ref{prop:EB1}]
Pick $z\in\dom(g)$. Without loss of generality, assume $z\notin\scrS$. Let $\bar{z}=\Pi_{\scrS}(z)$, so that there exists $p\in\NC_{\scrS}(\bar{z})$ satisfying $p=z-\bar{z}$. Hence, $p^{\ast}\eqdef \frac{z-\bar{z}}{\norm{z-\bar{z}}}$ is a unit norm element of $\NC_{\scrS}(\bar{z})$ satisfying $\inner{p^{\ast},z-\bar{z}}
=\norm{z-\bar{z}}.$
By definition of the convex tangent cone, we have $\frac{z-\bar{z}}{\norm{z-\bar{z}}}\in\TC_{\dom(g)}(\bar{z})$, and consequently, $p^{\ast}\in \TC_{\dom(g)}(\bar{z})\cap\NC_{\scrS}(\bar{z})$. Additionally $p^{\ast}\in\ball(0,1)$. By definition of weak sharpness, there exists $\xi^{\ast}\in\partial g(\bar{z})$ and $\tau >0$ such that 
\begin{equation*}
\tau p^{\ast}-\opF(\bar{z})-\xi^{\ast}\in[\TC_{\dom(g)}(\bar{z})\cap\NC_{\scrS}(\bar{z})]^{\circ}.
\end{equation*}
Hence, 
\begin{align*}
&\tau\inner{p^{\ast},z-\bar{z}}\leq\inner{\opF(\bar{z})+\xi^{\ast},z-\bar{z}} \leq \inner{\opF(\bar{z}),z-\bar{z}}+g(z)-g(\bar{z})\\
&\implies\tau\norm{z-\bar{z}}\leq H^{(\opF,g)}(z,\bar{z})\\
&\implies \tau\dist(z,\scrS)\leq H^{(\opF,g)}(z,\bar{z})\leq \sup_{y\in\dom(g)}H^{(\opF,g)}(z,y)=\Theta(z\vert\opF,g,\dom(g)).
\end{align*}
\end{proof}

\begin{proof}[Proof of Lemma \ref{lem:LB}]
Let $z^{*}\in\scrS_{1}=\zer(\opF_{1}+\partial g_{1}+\NC_{\scrS_{2}})\subset\scrS_{2}$. Then, there exists $p^{*}\in\NC_{\scrS_{2}}(z^{*})$ such that $-\opF_{1}(z^{*})-p^{*}\in\partial g_{1}(z^{*})$. By the convex sugradient inequality, this implies 
\begin{equation}\label{eq:b1}
\inner{\opF_{1}(z^{*}),z-z^{*}}+g_{1}(z)-g_{1}(z^{*})\geq \inner{-p^{*},z-z^{*}}\qquad\forall z\in\scrZ.
\end{equation}
Take $z\in \scrZ$ and let $\hat{z}=\Pi_{\scrS_{2}}(z)$. Thus, $\inner{p^{*},\hat{z}-z^*}\leq 0$,  resulting in $\inner{p^{*},\hat{z}-z}\leq\inner{p^{*},z^{*}-z}, $
which implies when combined with \eqref{eq:b1} 
\begin{align*}
\inner{\opF_{1}(z^{*}),z-z^{*}}+&g_{1}(z)-g_{1}(z^{*})\geq \inner{p^{*},\hat{z}-z}\geq -\norm{p^{*}}\cdot\norm{\hat{z}-z}=-\norm{p^{*}}\dist(z,\scrS_{2}).
\end{align*}
Hence, for all compact $\scrU_1\subset\dom(g_{1})$ with $z^*\in\scrU_1\cap\scrS_{1}\neq\emptyset$ and corresponding $p^{*}\in\NC_{\scrS_{2}}(z^{*})$, we conclude 
$\Theta_{\rm Opt}(z\vert\scrU_1\cap\scrS_1)\geq-B_{\scrU_1}\dist(z,\scrS_{2})$, where $B_{\scrU_1}\eqdef\norm{p^{*}}$. 

To show \eqref{eq:WS}, we can directly use Definition \ref{def:WS}, to conclude
\begin{align*}
\Theta_{\rm Feas}(z\vert\scrU_2\cap\scrS_2)=\sup_{z^{*}\in\scrU_2\cap\scrS_2}H^{(\opF_{2},g_{2})}(z,z^{*})\geq \frac{\alpha}{\rho}\dist(z,\scrS_{2})^{\rho}.
\end{align*}
 \end{proof}

\end{appendix}

\bibliographystyle{plain}
\bibliography{PenaltyDynamics}

 \end{document}